 \documentclass[smallextended]{amsart}
\usepackage[utf8]{inputenc}

\usepackage{comment}
\usepackage{geometry}
\usepackage{amsmath}
\usepackage{amsfonts}
\usepackage{amssymb}
\usepackage{enumitem}
\usepackage{fancyhdr}
\usepackage{float}
\usepackage{graphicx}
\usepackage[arrowdel]{physics}
\let \div \relax

\usepackage{subcaption}
\usepackage{tkz-euclide}
\usepackage{MyPackage}

\colorlet{red}{black} %un-comment to remove red highlights

\newtheorem{theorem}{Theorem}[section]
\newtheorem{prop}[theorem]{Proposition}
\newtheorem{defin}[theorem]{Definition}
\newtheorem{lemma}[theorem]{Lemma}
\newtheorem{rem}[theorem]{Remark}

\DeclareMathOperator{\E}{\mathbb{E}}
\DeclareMathOperator{\I}{\mathcal{I}}

\title[Inverse of divergence in randomly perforated domains]
{Inverse of divergence and homogenization of compressible Navier-Stokes equations in randomly perforated domains
}
\date{}
\author{Peter Bella \and Florian Oschmann}
\begin{document}
\maketitle

%\pagestyle{fancy}
%\fancyhf{}
%\fancyhead[CO]{P. Bella, F. Oschmann}
%\fancyfoot[C]{\thepage}
%\renewcommand{\headrulewidth}{0pt}

\begin{abstract}
We analyze the behavior of weak solutions to compressible viscous fluid flows in a bounded domain in $\R^3$, randomly perforated by tiny balls with random size. Assuming the radii of the balls scale like $\eps^\alpha$, $\alpha > 3$, with $\eps$ denoting the average distance between the balls, the problem homogenize to the same limiting equation. Our main contribution is a construction of the Bogovski\u{\i} operator, uniformly in $\eps$, without any assumptions on the minimal distance between the balls.
\end{abstract}

\section{Introduction}\label{sec:Introd}
The goal of this paper is to analyze the effective behavior of a compressible viscous fluid in randomly perforated domains. We consider a bounded domain $D \subset \R^3$ which for $\eps > 0$ is perforated by random balls $B_{\eps^{\alpha} r_i}(\eps z_i)$ with $\alpha > 3$,
and show that weak solutions to the Navier-Stokes equations in these perforated domains converge as $\eps \to 0$ to a weak solution to the same equation in $D$. Compared to the previous results, mostly restricted to the periodic arrangement of the holes~\cite{DieningFeireislLu,FeireislLu,Lu20} or at least assuming minimal distance between the holes being $\eps$~\cite{LuSchwarzacher}, we do not require any such assumptions on the minimal distance between the holes. 

As in the previous works on this topic, the key step in the proof of the homogenization result is the construction of the Bogovski\u{\i} operator (inverse of the divergence), bounded \emph{independently} of $\eps$. This operator is then used in a classical way to construct a test function, hence providing uniform estimates on the density and velocity of the fluid. 

\medskip

%\subsection{Known results}\footnote{This part about previous results should be extended (I can do it later)}
The question how small perforations in the domain influence the original equation has a long history. In the case of a perforation by periodically arranged balls in $d$ dimensions, Cioranescu and Murat~\cite{CioranescuMurat82} (see also~\cite{ConcaDonato88,Jing20}) studied effective behavior of the Poisson equation with zero Dirichlet boundary conditions on the balls. Denoting the distance between balls by $\eps$ and assuming the radii scale like $\eps^{\frac{d}{d-2}}$, they identified an additional Brinkman term (``A strange term coming from nowhere'') in the limiting equation. 

Heuristically, the homogeneous boundary conditions on the perforation push the solution $u$ towards $0$, with the strength related to the size of holes. Focusing on the case $d=3$, holes of size $\e^\alpha$, $\alpha > 3$ are too tiny to make any difference. If the balls are larger that $\eps^3$ (hence push $u$ stronger to $0$), solutions converge to $0$ as $\eps \to 0$, and only a rescaling by some negative power of $\eps$ may lead to a reasonable limiting problem. 

For the incompressible stationary Stokes and Navier-Stokes equations with periodic distribution of holes, Allaire \cite{Allaire90a,Allaire90b} gave a full description to all cases $\alpha>1$ and all dimensions $d\geq 2$. More precisely, for $d=3$ if $\alpha\in (1,3)$, which corresponds to the supercritical case of large particles, he obtained Darcy's law; for the critical case $\alpha=3$, an additional friction term occurs and gives rise to Brinkman's law. %Such a ''strange term coming from nowhere'' first arises in \cite{CioranescuMurat}, where the authors considered Poisson's equation. 
The subcritical case $\alpha>3$ corresponding to small particles leads to the same system of Stokes and Navier-Stokes equations. The case $\alpha=1$ was studied in \cite{Allaire89} for the steady incompressible Stokes system. % and in \cite{MasmoudiHom} for the time-dependent compressible Navier-Stokes system. %More recently, Giunti and Höfer showed in \cite{GiuntiHoefer19} that under the assumption of randomly distributed holes with random radii, the randomness does not affect the homogenization procedure.\\

%Such heuristics predicts the right behavior also in the case of effective fluid flows in porous media. In the case of incompressible fluids, modeled either by Stokes or Navier-Stokes equations, Allaire~\cite{Allaire90b,Allaire90a} (see also~\cite{Mikelic91}) analysed all three cases: supercritical ($\alpha < 3$) leading to Darcy's law, critical ($\alpha = 3$) giving the same equation with the additional Brinkman term related to friction, as well as the subcritical case $\alpha > 3$. For some recent results in the incompressible setting see, e.g.,~\cite{FeireislNamlNecasova16,Hillairet18}. 

The results on the effective behavior of \emph{compressible} fluids in perforated domains are much more recent. Masmoudi~\cite{Masmoudi02,Masmoudi07} considered compressible Navier-Stokes equations in the domain perforated by periodic balls with $\alpha = 1$, and obtained in the limit the Darcy's law. This result was later generalized by Feireisl, Novotn\'y and Takahashi~\cite{FeireislNovotnyTakahashi10} also to the case of the full Navier-Stokes-Fourier system, which besides density and velocity of the fluid takes also into account the fluid temperature. In the case of slightly smaller balls with $1 < \alpha < 3$, assuming simultaneous rescaling of the pressure (Low-Mach number limit) to avoid the need to study the ``cell problem'' with unknown density, H\"ofer, Kowalczyk, and Schwarzacher~\cite{SchwarzacherDarcy} showed convergence of the rescaled solution to the Darcy's law. Finally, for tiny balls with $\alpha > 3$ (the subcritical case) Feireisl and Lu~\cite{FeireislLu} considered stationary Navier-Stokes equations and showed convergence to the same equations in the domain without holes. This result was later improved to the case of more general adiabatic exponent in the pressure~\cite{DieningFeireislLu} as well as to the time-dependent case~\cite{LuSchwarzacher}. In all these works the perforation is assumed to be periodic, or at least the minimal distance between the holes is comparable with $\eps$. 

In this work we also consider only the subcritical case $\alpha > 3$, but with random arrangement of holes instead of the periodic one. Unless one additionally assumes that the holes in the random case may not lie close to each other, the key argument in the previous works, the Bogovki\u{\i} operator, can not be constructed as before. Instead, we show that while the holes can be close to each other, there exists a fixed number $N$ such that there are at most $N$ balls clustered together. Since $N$ is fixed, the construction of the Bogovki\u{\i} operator can be done.% {\color{red} REMOVE This number $N$ grows with $\frac{1}{\alpha-3}$, hence our argument for the construction of the Bogovki\u{\i} operator requires $\alpha > 3$. This is in contrast with the previous works~\cite{DieningFeireislLu,FeireislLu,LuSchwarzacher}, where for the Bogovki\u{\i} $\alpha \ge 3$ is sufficient and $\alpha = 3$ is excluded only later in the argument, where one needs to identify the limiting problem - for the this reason the critical case $\alpha = 3$ remains open even in the periodic setting.} 

Our inspiration how to approach the case of randomly perforated domains comes from a recent work of Giunti, H\"ofer, and Vel\'{a}zquez. In~\cite{GiuntiHoeferVelazquez18}, they considered a Poisson equations in a domain perforated by random balls of critical size, thus obtaining the Brinkman law in the limit. The main challenge is to understand possible clustering of the holes and control the capacity of those. In subsequent works, they also considered the incompressible Stokes problem~\cite{GiuntiHoefer19,Giunti20} as well as convergence to the Darcy's model in the supercritical situation~\cite{Giunti21}. Compared to these works our situation is simpler, since in the subcritical situation we can prove a deterministic upper bound on the size of clusters. Let us also mention that previously Beliaev and Kozlov~\cite{BeliaevKozlov96} also considered homogenization of Stokes equations in a supercritically perforated domain. In this work we tackle the situation $\alpha > 3$, while the case $1 < \alpha < 3$ leading to Darcy's law will be analyzed in the forthcoming work~\cite{BellaOschmann21}.

A different, though quite related, setting is of the flow of a colloidal suspension, i.e., of a fluid mixed with moving obstacles. This question traces back to one part of Einstein's PhD thesis~\cite{Einstein06}, where he formally derives effective viscosity of such suspension, assuming low-volume fraction of the obstacles. With the recent progress in the field of stochastic homogenization, this question was rigorously approached by several groups, first under the assumption of uniform separation of balls~\cite{DuerickxGloria21a,NiethammerSchubert20} and very recently under less restrictive assumptions~\cite{Duerinckx20,VaretHoefer20}.

\section{Setting and the Main Results}\label{sec:Model}

In this section we define the perforated domain, formulate the Navier-Stokes equations governing the fluid motion, and state the main results. We consider $D \subset \R^3$ being a bounded domain with a $C^2$ boundary. To simplify the probabilistic argument we farther assume the domain $D$ is star-shaped w.r.t.~the origin, i.e., for any $x \in D$ the segment $\{ \lambda x : \lambda \in [0,1] \} \subset D$. 

We model the perforation of $D$ using the Poisson point process, though the arguments can be easily generalized to a larger class of point processes. For an intensity parameter $\lambda > 0$, the Poisson point process is defined as a \emph{random} collection of points $\Phi = \{z_j\}$ in $\R^3$ characterized by the following two properties:
\begin{itemize}
 \item for any two measurable and \emph{disjoint} sets $S_1, S_2 \subset \R^3$, the random variables $S_1 \cap \Phi$ and $S_2 \cap \Phi$ are independent;
 \item for any measurable set $S \in \R^3$ and $k \in \mathbb{N}$ holds $\mathbb P(N(S) = k) = \frac{(\lambda |S|)^ke^{-\lambda |S|}}{k!}$,
\end{itemize}
where $N(S)=\# (S\cap\Phi)$ counts the number of points $z_j\in S$ and $|S|$ denotes measure of $S$. In addition to the random locations of the balls, modeled by the above Poisson point process, we also assume the balls have random size. For that, let $\mathcal{R} = \{r_i\} \subset [0,\infty)$ be another random process of independent identically distributed random variables with finite {\color{red}$m$-th moment, i.e., 
\begin{equation*}%\label{Conditionm}
 \E(r_i^m) < \infty \text{ for some $m>0$},
\end{equation*}}%
and which are independent of $\Phi$. In other words, to each point $z_j \in \Phi$ (center of a ball) we associate also a radius of the ball $r_j \in [0,\infty)$. The exact range of $m$ we can work with will be specified in Theorem \ref{thm:MainBog} below. The random process $(\Phi,\mathcal{R})$ on $\R^3 \times \R_+$ is called \emph{marked Poisson point process}, and can be viewed as a random variable $\omega \in \Omega \mapsto (\Phi(\omega),\mathcal{R}(\omega))$, defined on an abstract probability space $(\Omega,\mathcal{F},\mathbb{P})$. 

To define the perforated domain $D_\eps$, for {\color{red}$\alpha > 2$} and $\eps > 0$ we set 
\begin{equation}\label{def:Domain}
 \Phi^{\eps}(D) := \left\{ z \in \Phi \cap \frac1\eps D: \dist(\eps z,\partial D) > \eps\right\}, \quad D_\eps := D \setminus \bigcup_{z_j \in \Phi^{\eps}(D)} B_{\eps^{\alpha} r_j}(\eps z_j).
\end{equation}
To simplify the exposition and to avoid the need to analyze behavior near the boundary, we only removed those balls from $D$ which are not too close to the boundary $\partial D$. This is also a common assumption in the periodic situation, see, e.g., \cite[relation $(1.3)$]{FeireislLu}. The domain $D$ being star-shaped implies that $\Phi^\eps(D)$ are monotonically increasing as $\eps \to 0$. 
\begin{comment}
{\color{red}REMOVED, DONE IN SECTION 6 Condition~\eqref{Conditionm} on the size of radii of the perforations is not just needed for technical purposes, but it is in a sense an optimal assumption. Indeed, one can show that in the case $m=3/(\alpha-3)$, for almost every realization of points and radii there is a sequence of $\eps \to 0$ such that for each such $\eps$ the rescaled radius $\eps^\alpha r_j$ of the largest ball in $D$ is of size $\eps^3$, i.e., $r_j \sim \eps^{3-\alpha}$. While \emph{one} large ball of size $\eps^3$ does not necessarily mean that the system should behave as in the critical case (which is expected to lead to a law of Brinkman type), nevertheless the presence of such large ball might change some of the properties of the system. Moreover, in the case $m < 3/(\alpha-3)$, the size of the largest ball would scale like $\eps^\nu$ with $\nu<3$, and there might be many balls of size at least $\eps^3$. 
}%
\end{comment}

Our main result is the following existence result for a uniformly bounded Bogovski\u{\i} operator:

\begin{theorem}\label{thm:MainBog}
Let {\color{red}$\alpha > 2$,} $D\subset\R^3$ be a bounded star-shaped domain w.r.t.~the origin with $C^2$-boundary, and $(\Phi,\mathcal{R})=(\{z_j\},\{r_j\})$ be a marked Poisson point process with intensity $\lambda > 0$. {\color{red}We assume the radii $r_j \ge 0$ fulfil $\E(r_j^m)<\infty$ for some $m>3/(\alpha-2)$.} Then, {\color{red} for all $1<q<3$ which fulfil
\begin{align}\label{Conditionm}
\alpha-\frac3m>\frac{3}{3-q},
\end{align}}%
there exists a random almost surely positive $\eps_0 = \e_0(\omega)$ such that for $0<\e\leq\e_0$ there exists a bounded linear operator
\begin{align*}
\B_\e:L^q(D_\e) / \mathbb{R} \to W_0^{1,q}(D_\e;\R^3)
\end{align*}
with $D_\eps$ defined in~\eqref{def:Domain}, such that for all $f\in L^q(D_\e)$ with $\int_{D_\eps} f = 0$ {\color{red}
\begin{align*}
\div (\B_\e(f))=f \text{ in } D_\e,\quad \|\B_\e(f)\|_{W_0^{1,q}(D_\e)}\leq C \, \|f\|_{L^q(D_\e)},
\end{align*}}%
where the constant $C>0$ is independent of $\omega$ and $\e$.  
\end{theorem}

In other words, Theorem~\ref{thm:MainBog} provides a solution $\vb u$ to the equation $\div \vb u = f$ in $D_\eps$ with $\vb u|_{\partial D_\eps}=0$ such that {\color{red} $\| \nabla \vb u \|_{L^q(D_\eps)} \le \| \vb u \|_{W^{1,q}_0(D_\eps)} \le C\, \| f \|_{L^q(D_\eps)}$. }%

\medskip 
As an application for this result we show homogenization of compressible Navier-Stokes equations in perforated domains $D_\eps$. For $\eps > 0$, the unknown density $\rho_\eps$ and velocity $\vb u_\eps$ of a viscous compressible fluid are described by
\begin{equation}\label{eq:NS}
\begin{aligned}
\partial_t \rho_\e + \div(\rho_\e \vb u_\e)&=0\hspace{3.8em}&&\text{ in } (0,T) \times D_\e,\\%,\label{eq:conteq}\\
\partial_t(\rho_\e \vb u_\e) + \div(\rho_\e \vb u_\e\otimes \vb u_\e) + \nabla p(\rho_\e) &= \div\mathbb{S}(\nabla \vb u_\e) + \rho_\e \vb f+\vb g&&\text{ in } (0,T) \times  D_\e,\\%\label{eq:momentum}\\
\vb u_\e&=0\hspace{3.8em} &&\text{ on } (0,T) \times \partial D_\e,%\label{eq:bdryval}
\end{aligned}
\end{equation}
where $\mathbb{S}$ denotes the Newtonian viscous stress tensor of the form
\begin{align*}
\mathbb{S}(\nabla \vb u)=\mu \bigg(\nabla \vb u+\nabla^T \vb u-\frac23 \div(\vb u)\mathbb{I}\bigg)+\eta\div(\vb u)\mathbb{I},\quad \mu > 0, \quad \eta \ge 0,
\end{align*}
$p(\rho) = a\rho^{\gamma}$ denotes the pressure with $a > 0$ and the adiabatic exponent $\gamma \ge 1$, and $\vb f$ and $\vb g$ are external forces, which are for simplicity assumed to satisfy $\|\vb f\|_{L^\infty((0,T)\times\R^3;\R^3)} + \|\vb g\|_{L^\infty((0,T)\times\R^3;\R^3)} \le C$. 
We also fix the total mass 
%\begin{align*}
$\int_{D_\e} \rho_\e(0,\cdot)=\mass>0$
%\end{align*}
independently of $\e>0$, and supplement the equations with the initial conditions for $\rho_\eps$ and $\vb u_\eps$. 

While the existence of classical solutions to~\eqref{eq:NS} is known only in some special cases, the existence theory for weak solutions is quite developed~\cite{Feireisl04,Lions98,Novotny2004}. In particular, for fixed $\eps > 0$ the domain $D_\eps$ is smooth enough to grant existence of global weak solutions. The key point here is to overcome the lack of uniform estimates in $\eps$ on the smoothness of $D_\eps$, in particular to obtain uniform bounds on the solution, which in this setting are usually obtained using a bounded Bogovki\u{\i} operator. More precisely, Theorem~\ref{thm:MainBog} together with a simple existence result for the cut-off function (see Lemma~\ref{lm:cutoff}) are the only points in the arguments~\cite{DieningFeireislLu,FeireislLu,LuSchwarzacher}, where the information on the structure of the perforation in $D_\eps$ is being used. 

In the following we state one of the implications of Theorem~\ref{thm:MainBog}, the corresponding precise formulation as well as the definition of the finite energy weak solutions in the case of periodic perforation being~\cite[Definition 1.3, Theorem 1.6]{LuSchwarzacher}:

\begin{theorem}\label{thm:main}
{\color{red}Assume $\alpha>3$.} Let $D\subset \R^3$ be a bounded star-shaped domain w.r.t.~the origin with $C^2$-boundary and let $(\Phi,\mathcal{R})=(\{z_j\},\{r_j\})$ be a marked Poisson point process with intensity $\lambda > 0$, and $r_j \ge 0$ with {\color{red}$\mathbb{E}(r_j^M) < \infty$, $M=\max\{3,m\}$, where $m>3/(\alpha-3)$.} Farther let
\begin{align*}
\mass>0,\; \gamma>6.
\end{align*}
For $0 < \eps < 1$ let $[\rho_\eps,\vb u_\e]$ be a family of finite energy weak solutions for the no-slip compressible Navier-Stokes equations~\eqref{eq:NS} in $(0,T)\times D_\e$ with controlled initial conditions
\begin{equation*}
 \rho_\e(0,\cdot) := \rho_{0,\eps}, \ \vb u_\e(0,\cdot) := \vb u_{0,\eps}, \ \sup_{0 < \eps < 1} (\|\rho_{0,\eps}\|_{L^\gamma(D_\eps)} + \|\vb u_{0,\eps}\|_{L^3(D_\eps)} ) = B < \infty,
\end{equation*}
with $D_\e$ as in~\eqref{def:Domain}. Then for almost every $\omega \in \Omega$ there exists $\eps_0=\eps_0(\omega) > 0$, such that the following holds: 
There exists a constant $C(B)>0$, which is independent of $\e$, such that 
\begin{align*}
\sup_{\e\in (0,\eps_0)} ( \|\rho_\e\|_{L^\infty(0,T;L^{\gamma}(D_\eps))}+ \|\rho_\e\|_{L^{\frac{5\gamma}{3}-1}((0,T) \times D_\e)} + \|\vb u_\e\|_{L^2(0,T;W^{1,2}_0(D_\e))}) \leq C
\end{align*}
and, up to a subsequence, the zero extensions satisfy
\begin{align*}
\tilde{\rho}_\e\overset{\ast}\weak \rho \text{ in } L^\infty(0,T;L^\gamma(D)),\quad \tilde{\vb u}_\e\weak \vb u \text{ in } L^2(0,T;W^{1,2}_0(D)),
\end{align*}
where the limit $[\rho,\vb u]$ is a renormalized finite energy weak solution to the problem \eqref{eq:NS} in the limit domain $D$ provided $\frac{\gamma -6 }{2\gamma - 3}\alpha > 3$. 
\end{theorem}

The restriction $M\geq 3$ is made in order to construct suitable cut-off functions, as will be clear from the proof of Lemma \ref{lm:cutoff} later on. Using Theorem~\ref{thm:MainBog} and Lemma~\ref{lm:cutoff}, which are the only two spots in the proof where the structure of the perforation plays any role, the proof of Theorem~\ref{thm:main} follows verbatim as in~\cite{LuSchwarzacher}. To manifest how Theorem~\ref{thm:MainBog} and Lemma~\ref{lm:cutoff} are actually applied, in Section~\ref{sec:ApplNSE} we will formulate a similar (but simpler) homogenization statement for the stationary case and sketch its proof. 

\subsection{Notation}
Through the whole paper, we use the following notation:
\begin{itemize}
\item $(\Omega, \mathcal{F},\P)$ is the probability space associated to the marked point process $(\Phi,\mathcal{R})$.
\item $L_0^p(D):=\{f\in L^p(D): \int_D f=0\}$
\item $|S|$ denotes the Lebesgue measure of a measurable set $S\subset\R^3$.
\item For a function $f$ with domain of definition $D$ or $D_\eps$, we denote by $\tilde{f}$ the zero extension to $\R^3$, that is, we define
\begin{align*}
\tilde{f}=f \text{ in } D \textrm{ or } D_\e,\quad \tilde{f}=0 \text{ in } \R^3\setminus D.
\end{align*}
\item Boxes are sets of the form $A_x \times A_y \times A_z$, where $A_x,A_y,A_z \subset \R$ are intervals. 
\item For a factor $\lambda >0$ and a set $M\subset\R^d$ we define $\lambda M := \{ \lambda x : x \in M \}$.
\item For two sets $M,N\subset\R^3$, we set $\displaystyle\dist(M,N)=\inf_{x\in M, y\in N} |x-y|$, where $|x|$ is the usual Euclidean norm, and $\displaystyle\dist_\infty(M,N)=\inf_{x\in M, y\in N} \|x-y\|_\infty=\inf_{x\in M, y\in N} \max_{1\leq i\leq 3} |x_i-y_i|$.
\item We write $a\lesssim b$ whenever there is a constant $C>0$ that does not depend on $\e, a$ and $b$ such that $a\leq C\, b$. The constant $C$ might change its value whenever it occurs.
\end{itemize}

Moreover, if no ambiguity occurs, we denote the function spaces as in the scalar case even if the functions are vector- or matrix-valued, e.g., we write $L^p(D)$ instead of $L^p(D;\R^3)$.\\

{\bf Organization of the paper:} The paper is organized as follows. In the next section we formulate the probabilistic statements (Theorem~\ref{thm:MainProbThm} and Proposition~\ref{prop:cubeN}) as well as the analytical framework (Lemma~\ref{lem:O-EpsJohn}) needed for the construction of the Bogovski\u{\i} operator (Theorem~\ref{thm:MainBog}). The proofs of these results are content of Section~\ref{sec:Prob} (probabilistic part) and Section~\ref{sec:BogOp} (analytical part). The last section is devoted to a quick sketch of the homogenization result. 

%In Section \ref{sec:Model} we state some technical lemmata which are crucial to prove the norm bound for the desired Bogovski\u{\i} operator. Section \ref{sec:BogOp} includes the proofs of these results. Section~\ref{sec:ApplNSE} is devoted to the application to homogenization of the Navier-Stokes equations for a compressible fluid. Section \ref{sec:Prob} contains some probabilistic results and their proofs

\section{Ingredients for the proof of Theorem~\ref{thm:MainBog}}\label{sect:ingr}

The proof of Theorem~\ref{thm:MainBog} consists of two parts: stochastic and analytical. The stochastic result, Theorem~\ref{thm:MainProbThm}, states that for small enough (depending on $\omega \in \Omega$) $\eps > 0$ the balls with radii $\eps^{\alpha} r_j$ are disjoint and actually little bit separated. The previous construction of the uniformly bounded Bogovki\u{\i} operator requires a boundary layer of size third root of the radius of the balls without hitting other balls -- this is where the condition $\alpha \ge 3$ enters. Since in a generic random arrangement of balls the balls are \emph{not} that much separated, we relax this assumption by replacing one ball with finitely many balls. More precisely, we show that there exists a deterministic number $N=N(\alpha)$ such that we can group balls into clusters of size at most $N$ so that the clusters stay separated from each other. %{\color{red}REMOVE Since this number $N$ is inversely related to $\alpha - 3$, we get finite $N$ only if $\alpha > 3$.} 

\begin{theorem}\label{thm:MainProbThm}
{\color{red}Let $\alpha>2$ and $\lambda>0$ be the intensity of a marked Poisson point process $(\Phi,\mathcal{R}) = (\{z_j\},\{r_j\})$ with $r_j \ge 0$ and $\E(r_j^m)<\infty$, where $m>0$ satisfies
\begin{align*}
m>\frac{3}{\alpha-2}.
\end{align*}
Let $0<\delta< \alpha-1-\frac3m$, $\kappa \in (\max(1,\delta),\alpha-1-\frac3m)$, and $\tau \ge 1$. Then there exists a random variable $\eps_0$, which is almost surely positive, satisfying: }%
%, $\Phi^\e(D)$ be defined as in \eqref{def:PhiEps}, $\{r_i\}_{i\in\N}\subset [0,\infty)$ be iid.~random variables which satisfy \eqref{eq:assumptionRadii}, $\alpha>3$, $\kappa\in (1,\alpha-2]$ and $\tau\geq 1$. Further, let $\delta=(\alpha-3)/3>0$. Then there exists an $\e_0(\omega)>0$ such that 
%for all $0<\e\leq \e_0$, we have:
\begin{enumerate}%[label={(\arabic*)}]
\item For every $0<\e\leq \e_0$ holds: \label{veta31jedna}
\begin{equation*}
\max_{z_i\in\Phi^\e(D)} \tau \e^\alpha r_i\leq \e^{1+\kappa} 
\end{equation*}
and for every $z_i,z_j\in\Phi^\e(D),\, z_i\neq z_j$ 
\begin{equation*}
\B_{\tau\e^{1+\kappa}}(\e z_i)\cap B_{\tau\e^{1+\kappa}}(\e z_j)=\emptyset.
\end{equation*}

\item Let 
\begin{align}\label{def:N(delta)}
%N:=N(\delta):=8\big(\left\lceil 2^{1+\delta}\right\rceil +1\big)\bigg(2+\left\lceil\frac{1}{\delta}\right\rceil\bigg).
N:=N(\delta):=8 \bigg(2+\left\lceil\frac{1}{\delta}\right\rceil\bigg).
\end{align}
Then for each $0 < \eps \le \eps_0$ there are finitely many open boxes $\{I_i^\eps\} \subset D$ satisfying: 
% is a finite index set $\mathcal{I}^\e$ such that for all $\nu\in\mathcal{I}^\e$ there are open boxes $O_\nu^\e$ and $I_\nu^\e\subset O_\nu^\e$ such that:
\begin{enumerate}
\item The boxes $I_i^\eps$ cover the balls, i.e., for any $z \in\Phi^\e(D)$ we have $B_{\e^{1+\kappa}}(\e z)\subset \bigcup_i I_i^\e$. \label{thm31labela}
\item Any box $I_i^\e$ contains at most $N$ points from $\eps \Phi^\e(D)$. \label{thm31labelb}
\item Balls are well inside the box: for $\e z \in I_i^\e$ holds $\dist(B_{\e^{1+\kappa}}(\e z),\d I_i^\e)\geq \frac{1}{16N}\e^{1+\delta}$. \label{thm31labelc}% for all $\e z_i\in I_\nu^\e$.
%\item $\dist(\d O_\nu^\e, \d I_\nu^\e)\geq \e^{1+\delta}/(2N)^2$.
\item Any two distinct boxes $I_i^\eps$ and $I_j^\eps$ are well separated: $\dist_\infty (I_i^\e,I_j^\e) \ge \frac{1}{4N}\e^{1+\delta}$. \label{thm31labeld}
%For all $\nu,\tilde{\nu}\in\mathcal{I}^\e$ with $\nu\neq\tilde{\nu}$, the boxes $O_\nu^\e$ and $O_{\tilde{\nu}}^\e$ are essentially disjoint, that is, $|O_\nu^\e\cap O_{\tilde{\nu}}^\e|=0$.
\item The shortest side of $I^\e_i$ is at least $\frac{1}{2N}\eps^{1+\delta}$ while the longest side is at most $\eps^{1+\delta}$. 
%ratio between the longest and shortest side length of $I^\e_i$ is bounded by $N$.
\label{thm31labele}%\footnote{needs to modify - shortest at least ..., longest at most ...}
\end{enumerate} \label{lem:GroupingBalls}
\end{enumerate}
\end{theorem}

The proof of the second part of Theorem~\ref{thm:MainProbThm} uses that for $0 < \e \le \eps_0$ any cube with side length $\eps^{1+\delta}$ contains at most $N$ points from the Poisson point process. This can hold only if $\delta > 0$, since the number of points in a cube of size $\eps^{1+0}$ is Poisson distributed, i.e., any number of points appears there with small but positive probability. 

\begin{prop}\label{prop:cubeN}
Let $d\geq 1$, $\delta>0$ be fixed, and let $\{z_j\}\subset\R^d$ be points generated by a Poisson point process of intensity $\lambda>0$. In addition, let $D \subset \R^d$ be a bounded star-shaped domain. Then there exists a deterministic constant $N(\delta, d)\in\N$ and a random variable $\e_0(\omega,\lambda,D)$, which is almost surely positive, such that for all $0<\e\leq\e_0$ and any $x \in \R^d$ the cube $x + [0,\eps^{1+\delta}]^d$ contains at most $N$ points from $D \cap \eps \Phi$.
\end{prop}

To construct the Bogovski\u{\i} operator $\B_\eps$ in $D_\eps$ from Theorem~\ref{thm:MainBog} we use local Bogovski\u{\i} operators for each box $I_i^\e$ to modify the Bogovski\u{\i} operator in $D$. Instead of making explicit construction in each box $I_i^\eps$, we invoke a general result on the existence of Bogovski\u{\i} operator~\cite{AcostaDuran06} for a class of domains (so-called John domains) and show that each box $I_i^\eps$ minus the balls is a John domain -- for this the outcomes of Theorem~\ref{thm:MainProbThm} will be crucial. In particular, we need that there are at most $N$ balls in one box, the balls are not close to each other and they are tiny, compared to the size of the box. 

%The proofs of Theorem \ref{thm:MainProbThm} \ref{lem:e2disjoint} and \ref{thm:MainProbThm} \ref{lem:finitemanyballs} will be given in a more general setting in Section \ref{sec:Prob}. To state our final outcome, we need the notion of so-called John domains. We will give a more detailed description in section \ref{sec:ProofJohn}.

\begin{defin}\label{def:John}
For a constant $c>0$, a domain $U\subset \R^d$ is said to be a $c-$John domain if there exists a point $x_0\in U$ such that for any point $x\in U$ there is a rectifiable path $\Gamma:[0,\ell]\to U$ which is parametrized by arc length with
\begin{align*}
\Gamma(0)=x,\quad \Gamma(\ell)=x_0,\quad \forall t\in [0,\ell]: |\Gamma(t)-x|\leq c\, \dist(\Gamma(t),\d U).
\end{align*}
\end{defin}
The following lemma states that any $I_i^\eps$ is a $c-$John domain:
\begin{lemma}\label{lem:O-EpsJohn}%\footnote{2change: not refering to thm...}
Under the assumptions of Theorem~\ref{thm:MainProbThm} {\color{red}for fixed
\begin{align}\label{eq:DefDelta}
0<\delta<\frac{\alpha-2-\frac3m}{2},
\end{align}}%
let $0 < \eps \le \eps_0$. Then for every box $I_i^\eps$ constructed in Theorem~\ref{thm:MainProbThm}, the domain 
\begin{equation}
U:=I^\eps_i \setminus \bigcup\limits_{z_j\in \eps^{-1} I_i^\e \cap \Phi^\eps(D)} B_{\e^\alpha r_j}(\e z_j)
\end{equation}
% Let $\e_0>0$ and $O_\nu^\e$ be as in Theorem \ref{thm:MainProbThm} \ref{lem:GroupingBalls}. Then, for all $0<\e\leq\e_0$ and all $\nu\in\mathcal{I}^\e$, the domain
% \begin{align*}
% U:=O_\nu^\e\setminus \bigcup\limits_{z_i\in O^\e} B_{r_i\e^\alpha}(\e z_i)
% \end{align*}
is a $c-$John domain with $c=c(N)$, where $N$ is defined in~\eqref{def:N(delta)}. %, where the constant $c$ just depends on the number $N$ defined in \eqref{def:N(delta)}.

In particular, for any $1 < q < \infty$ there exists a uniformly bounded Bogovki\u{\i} operator $\B_U : L^q_0(U) \to W^{1,q}_0(U)$, i.e., there exists a constant $C$, independent of $\eps$, such that for any $f \in L^q_0(U)$%\footnote{Flo, I did not understand your comment $\partial_\eps ... = 0$.}
\begin{align*}
\div(\B_{U}(f))=f,\quad ||\B_U(f)||_{W_0^{1,q}(U)} \leq C ||f||_{L_0^q(U)}.
\end{align*}
\end{lemma}

\section{The probabilistic results}\label{sec:Prob}
The goal of this section is to prove the stochastic part of the result, Theorem~\ref{thm:MainProbThm}, second part of which is based on Proposition~\ref{prop:cubeN} about the distribution of the random points, modeled by the Poisson point process. Fixing $\delta > 0$, this proposition states that for $\eps$ small enough, for any cube of side length $\eps^{1+\delta}$ inside a fixed domain $D$ there are at most $N=N(\delta,d)$ of the rescaled points $\eps z$ in the cube. The heuristic explanation of this is as follows: assuming we only need to consider a disjoint set of cubes and fixing $\eps>0$, the number of cubes in $D$ which we have to consider scales like $\frac{1}{\eps^{(1+\delta)d}}$. At the same time, the probability of one cube of side length $\eps^{1+\delta}$ having more than $N$ points scales in the case of the Poisson point process like $(\frac{\eps^{(1+\delta)d}}{\eps^d})^N = \eps^{\delta N d}$. Hence, choosing $N$ large enough so that $\frac{1}{\eps^{(1+\delta)d}} \eps^{\delta N d} \ll 1$ should lead to the result. 

\begin{proof}[Proof of Proposition~\ref{prop:cubeN}]
We start with a special case, which will be later used to prove the general case:

\medskip\noindent
{\bf Claim:} There exists $N_1 \in \mathbb{N}$ and an a.s.~positive random variable $\eps_0(\omega)$ such that for any dyadic $\eps = 2^{-l}$ smaller than $\eps_0$, any half-closed cube $Q_{\eps,z} = \eps^{1+\delta} z + [0,\eps^{1+\delta})^d$, $z \in \mathbb{Z}^d$, contains at most $N_1$ points from $\frac{\eps}{2}\Phi(\omega) \cap D$.

\medskip
If rescaled up by a factor $2$ the claim says that in a cube with side length $(2\eps)^{1+\delta}$ there are at most $N_1$ points, and we are considering points (more precisely cubes) inside $2D$ instead of $D$ only. The reason for this choice will be clear later in the proof. 

% First, let us observe that it is enough to prove the proposition for dyadic values of $\eps = 2^l$, i.e. that there exists $N_1 \in \N$ and a random a.s. positive $\eps_0$ such that for any $l \ge \frac{\log(\eps_0)}{\log(2)}$ and $\eps_l := 2^{-l}$ any cube with side length $\eps_l^{1+\delta}$ contains at most $N_1$ points from $\eps_l \Phi \cap D$. 
% 
% Indeed, assuming this, for any $\eps$ satisfying $\e_l \le \eps < e_{l-1} \le \eps_0$ any cube $Q_\e$ with side length $\eps^{1+\delta}$ can be covered by at most $\big(\big\lceil 2^{1+\delta} \big\rceil+1\big)^d$ many cubes $Q_{\eps_l,i}$ of side length $\e_l^{1+\delta}$. By the above assumption any cube 
% %
% Assuming that any $C_{\e_l}^i$ contains no more that $N_1$ points, we have that the number of points inside $C_\e$ is bounded by
% \begin{align*}
% \#\{z_j\in C_\e\}\leq \sum_i \#\{z_j\in C_{\e_l}^i\}\leq \big(\big\lceil 2^{1+\delta} \big\rceil+1\big)^d\cdot N_1.
% \end{align*}
% %
% Finally, since for any $y\in\R^d$, a shifted cube $y+C_\e$ hits at most $2^d$ cubes $C_\e$, we define $N$ to be
% \begin{align*}
% N:=2^d\big(\big\lceil 2^{1+\delta} \big\rceil+1\big)^d N_1=2^d\big(\big\lceil 2^{1+\delta} \big\rceil+1\big)^d\bigg(2+\left\lceil\frac1\delta\right\rceil\bigg)
% \end{align*}
% 
% to finish the proof in the general case.\\
\medskip
% \noindent
To show the claim, for any such cube $Q_{\e,z}$ with side length $\e^{1+\delta}$ and $N_1\in \N$ we consider an event
\begin{align*}
\mathcal{A}^\e(Q_{\e,z}):=\{\omega \in \Omega: \text{there are at least $N_1$ points from $\frac\eps 2\Phi(\omega) \cap D$ in $Q_{\e,z}$}\}.
\end{align*}
Recall that for any measurable bounded set $S\subset\R^3$, we denote by $N(S)=\#(S\cap \Phi)$ the number of random points in $S$. Since the points $\Phi$ are Poisson-distributed, we have for any $n\in\N$
\begin{align*}
\P(N(S)=n)=e^{-\lambda |S|}\frac{(\lambda |S|)^n}{n!}.
\end{align*}
Hence, using $|Q_{\eps,z}| = \eps^{(1+\delta)d}$, we get
\begin{align*}
\P(\mathcal{A}^\e(Q_{\eps,z})) &\le \sum_{k = N_1}^\infty \P(N(\frac{2}{\eps}Q_{\eps,z}) = k) = e^{-\frac{\lambda}{(\e/2)^d}|Q_{\e,z}|}\sum_{k=N_1}^\infty \frac{\big(\frac{\lambda}{(\e/2)^d}|Q_{\e,z}|\big)^k}{k!}
\\
&\leq \frac{ \left( \lambda (\eps/2)^{-d} |Q_{\eps,z}|\right)^{N_1}}{N_1!} = \frac{(\lambda2^d)^{N_1}}{N_1!} \e^{\delta d N_1},
\end{align*}
where we used that for every $x > 0$
%Since the points are distributed by a Poisson point process, using
\begin{equation}\label{eq:exp}
e^{-x}\sum_{k\geq n} \frac{x^k}{k!}=\frac{x^n}{n!} e^{-x}\sum_{k\geq 0} \frac{x^kn!}{(n+k)!}\leq \frac{x^n}{n!}e^{-x}\sum_{k\geq 0} \frac{x^k}{k!}=\frac{x^n}{n!}.
\end{equation}
Since the domain $D$ is bounded, we can cover $D$ with less than $C(D) \eps^{-(1+\delta)d}$ disjoint cubes $Q_{\eps,z}$ of the form as above. The previous argument then implies
\begin{align}\label{eq:P(capAe)}
\P\bigg(\!\bigcap_{z} (\mathcal{A}^\e(Q_{\eps,z}))^c\!\bigg)\!=\!\prod_{z} (1-\P(\mathcal{A}^\e(Q_{\eps,z})) \!\geq \bigg(1-\frac{\lambda^{N_1}2^{dN_1}}{N_1!}\e^{\delta d N_1}\bigg)^{C(D)\e^{-d(1+\delta)}},
\end{align}
where $\mathcal A^c:=\Omega\setminus \mathcal A$ denotes the complementary event. The cubes $\{Q_{\eps,z}\}_{z \in \Z^d}$ are disjoint and so the events $\{\mathcal{A}^\e(Q_{\eps,z})\}_{z}$ as well as their complements are independent, what we used above. 
Choosing $N_1:=2+\big\lceil \frac1\delta\big\rceil$ we see that $N_1>(1+\frac1\delta)$ so that $\delta d N_1 - d(1+\delta) > 0$. 
% , in particular we have $\lim_{\e\to 0} \e^{\delta d N_1}\e^{-C(D)d(1+\delta)}=0$ and hence
% \begin{align*}
% \lim_{\e \to 0} \P\bigg(\bigcap_{z} (\mathcal{A}^\e(Q_{\eps,z}))^c\bigg) \le \limsup_{\e \to 0} \bigg(1-\frac{\lambda^{N_1}}{N_1!}\e^{\delta d N_1}\bigg)^{\e^{-C(D)d(1+\delta)}} = 0.
% \end{align*}
%
For $l \in \N$ let 
\begin{align*}
B_l:=\bigcup_{z} \mathcal{A}^{2^{-l}}(Q_{2^{-l},z})=\{\omega \in \Omega : &\textrm{ one of the dyadic cubes $Q_{2^{-l},z}$ contains}\\
&\textrm{ at least $N_1$ points from $2^{-l-1} \Phi \cap D$}\}
\end{align*}
and observe that \eqref{eq:P(capAe)} implies with $\eps = 2^{-l}$
\begin{equation}
\begin{aligned}\label{eq:P(Be)}
\P(B_l)&=1-\P\big(B_l^c\big)=1-\P\bigg(\bigcap_{z} (\mathcal{A}^{2^{-l}}(Q_{2^{-l},z}))^c\bigg)
\\
&\leq 1-\big(1-C_1 2^{-ld\delta N_1}\big)^{C(D)2^{ld(1+\delta)}},
\end{aligned}
\end{equation}
where $C_1(\lambda,d,N_1)$ is the fraction appearing on the right-hand side of \eqref{eq:P(capAe)}. 
%For abbreviation and legibility, from now on, we set
% \begin{align*}
% p:=2^{-d\delta N_1}<1<2^{C(D)d(1+\delta)}=:q,
% \end{align*}
% which yields
% \begin{align*}
% \P(B_l)\leq 1-(1-C_1p^l)^{q^{l}}.
% \end{align*}

To apply the Borel-Cantelli lemma, it is enough to show $\sum_{l=0}^\infty \P(B_l) < \infty$. 
% , which would follow from 
% \begin{align*}
% \sum_{l=0}^\infty \bigg[1-\big(1-C_1p^l\big)^{q^{l}}\bigg]<\infty.
% \end{align*}
% By definition of $N_1$, we have
% \begin{align*}
% p\leq 2^{-d(1+2\delta)}=\frac{1}{q2^{d\delta}},
% \end{align*}
Using Bernoulli's inequality $(1-x)^t\geq 1-tx$, which holds since $x = C_12^{-ld\delta N_1} \le 1$ for $l\ge l_0$ for large enough $l_0 \in \N$, we see
\begin{align*}
\sum_{l=l_0}^\infty \P(B_l) &\le \sum_{l=l_0}^\infty \bigg[1-\big(1-C_1 2^{-ld\delta N_1}\big)^{C(D)2^{ld(1+\delta)}}\bigg] 
\\
&\leq \sum_{l=l_0}^\infty C_1 C(D) 2^{-ld(\delta N_1 - (1+\delta))} < \infty,
\end{align*}
where in the last inequality we used that by definition of $N_1$ we have $\delta N_1 - (1+\delta) > 0$. The Borel-Cantelli lemma implies
\begin{align*}
\P\big(\limsup_{l\to\infty}B_l\big)=0,
\end{align*}
meaning that almost surely there is an $\e_0(\omega)>0$ such that for all $0<\e_l=2^{-l}\leq\e_0$, any cube $Q_{2^{-l},z}$ contains not more than $N_1$ points from $\frac \eps 2 \Phi \cap D$, thus proving the claim. 

\medskip
To show the general case, for $\omega \in \Omega$ we consider $\eps_0(\omega)$ coming from the claim. W.l.o.g.~we assume $\eps_0 = 2^{-l_0}$ for some $l_0 \in \N$ (otherwise replace $\eps_0$ with the largest smaller power of $2$). To finish the proof we need to show that for any $0 < \eps \le \eps_0$, any cube $Q_\eps = x + [0,\eps^{1+\delta}]^d$, $x \in \R^d$, contains at most $N$ points from $D \cap \eps \Phi(\omega)$. Let $0 < \eps < \eps_0$ and $Q_\e = x + [0,\eps^{1+\delta}]^d$ be chosen arbitrary, and let $N := 2^d N_1$. Let $l \ge l_0$ be the unique $l$ such that $2^{-(l+1)} \le \eps < 2^{-l}$. 

Observe that for $\lambda > 0$ we have $\# (Q_\e \cap \eps \Phi) = \# (\lambda Q_\e \cap \lambda \eps \Phi)$, where $\lambda Q_\e = \{ \lambda x : x \in Q_\e \}$, 
which together with star-shapedness of $D$ yields for $\lambda = \frac{2^{-(l+1)}}{\eps} \in (0,1]$
\begin{equation*}
\# (Q_\e \cap \eps \Phi \cap D) = \# (\lambda Q_\e \cap \lambda \eps \Phi \cap \lambda D) \le \# (\lambda Q_\e \cap 2^{-(l+1)} \Phi \cap D).
\end{equation*}
We now cover $\lambda Q_\eps$ with (at most) $2^d$ cubes $Q_{2^{-l},z}$. Observe that even if $\lambda Q_\eps$ is closed and $Q_{2^{-l},z}$ are only half-closed, the covering is possible since $\lambda \eps = 2^{-(l+1)} < 2^{-l}$. In particular, the claim implies that any $Q_{2^{-l},z}$ contains at most $N_1$ points from $\frac{2^{-l}}{2}\Phi \cap D$, thus implying that $\lambda Q_\eps$, being covered by at most $2^d$ cubes $Q_{2^{-l},z}$, contains at most $2^d N_1$ points from $2^{-(l+1)}\Phi \cap D$. This together with the last display implies $\# (Q_\e \cap \eps \Phi \cap D) \le 2^d N_1 = N$, thus concluding the proof of the proposition. 
\end{proof}

The first part of Theorem~\ref{thm:MainProbThm} is based on the following Strong Law of Large Numbers, which previously appeared as \cite[Lemma C.1]{GiuntiHoefer19}:

\begin{lemma}\label{lem:SLLN}
Let $d\geq 1$ and $(\Phi,\mathcal{R})=(\{z_j\},\{r_j\})$ be a marked Poisson point process with intensity $\lambda>0$. Assume that the marks $\{r_j\}$ are non-negative i.i.d.~random variables independent of $\Phi$ such that ${\E(r_j^m) <\infty}$ for some $m > 0$. Then, for every bounded set $S\subset\R^d$ which is star-shaped with respect to the origin, we have almost surely
\begin{align*}
\lim_{\e\to 0} \e^d N(\e^{-1} S)=\lambda |S|, \qquad \lim_{\e\to 0} \e^d \hspace{-0.3em} \sum_{z_j \in \eps^{-1} S} r_j^m=\lambda \E( r^m)|S|.
\end{align*}
%\footnote{change in the second formula in the sum}
\end{lemma}

\begin{rem}\label{rmk:SLLN} 
 Assuming the boundary of the set $S$ from the previous lemma is not too large, the same argument also shows
\begin{equation}\label{eq:ergodic}
\lim_{\e\to 0} \e^d \hspace{-0.3em} \sum_{z_j \in \Phi^{\eps}(S)} r_j^m=\lambda \E( r^m)|S|.
\end{equation}
In particular, it is enough that $S$ has as $D$ a $C^2$-boundary.  
\end{rem}

Using this remark as well as Proposition~\ref{prop:cubeN} we can prove Theorem~\ref{thm:MainProbThm}:  

\begin{proof}[Proof of Theorem~\ref{thm:MainProbThm}.]
{\bf Part (1):} We start with the first part of the theorem, which actually holds for any dimension $d \ge 1$, $\alpha > 2$, $m > \frac{d}{\alpha - 2}$, and {\color{red}$\kappa \in (1,\alpha - 1 - \frac d m)$.}

%By Lemma~\ref{lem:SLLN} 
Using~\eqref{eq:ergodic} and the choice of $\kappa$, we have for almost all $\omega$ %and $\e=\e(\omega)>0$ small enough
\begin{align*}
\limsup_{\eps \to 0} \eps^{\frac dm} \max_{z_i\in\Phi^\e(D)} r_i \leq \limsup_{\eps \to 0} \e^\frac{d}{m} \biggl( \sum_{z_i\in\Phi^\e(D)} r_i^m \biggr)^{\frac 1m} \leq [\lambda \E(r^m)|D|]^{\frac1m}. %\leq C \e^{1+\kappa}.
\end{align*}
{\color{red}This implies for $\e>0$ small enough
\begin{equation}
 \max_{z_i\in\Phi^\e(D)} \tau \e^\alpha r_i \leq 2\tau \e^{\alpha-\frac{d}{m}}[\lambda \E(r^m)|D|]^{\frac1m} \le \eps^{1+\kappa},
\end{equation}
the last inequality coming from $\alpha - \frac{d}{m} > \kappa+1$, and therefore being true for $\eps$ possibly even smaller.}

%This estimate immediately shows that for $\kappa<\alpha-1-d/m$ and $\e>0$ small enough, one may choose $C=1$.\\

To show two balls do not intersect we consider an event
\begin{align*}
A_\tau^\e:=\{\omega \in \Omega : \text{ there are 2 intersecting balls in } \{ B_{\tau\e^{1+\kappa}}(\e z)\}_{z \in \Phi^\eps(D)} \}
\end{align*}
and it is enough to show
\begin{align}\label{eq:BorCant1}
\P\bigg(\bigcap_{\e_0>0}\bigcup_{\e\leq\e_0} A_\tau^\e\bigg)=0.
\end{align}
We reduce this to the case of dyadic $\eps$, by showing 
\begin{align}\label{eq:BorCant2}
\P\bigg(\bigcap_{l_0\geq 1}\bigcup_{l\geq l_0} A_{\bar{\tau}}^{\e_l}\bigg)=0,
\end{align}
where $\e_l=2^{-l}$ and $\bar{\tau}=2^{1+\kappa}\tau$. 

Indeed, let $l\in\N$ be such that $\e_{l+1}\leq \e <\e_l$. 
% Then it is clear that
% \begin{align}\label{eq:PhiSubsetPhi_l}
% \Phi^\e(D)\subset\Phi^{\e_{l+1}}(D).
% \end{align}
Now suppose $z_i,z_j\in\Phi^\e(D)$, $z_i \neq z_j$ such that
\begin{align*}
B_{\tau\e^{1+\kappa}}(\e z_i)\cap B_{\tau\e^{1+\kappa}}(\e z_j)\neq\emptyset.
\end{align*}
Then
\begin{align*}
\e_{l+1}|z_i-z_j|\leq \e |z_i-z_j|\leq 2\tau\e^{1+\kappa}\leq 2\tau\e_l^{1+\kappa} = 2 \tau(2\e_{l+1})^{1+\kappa} = 2\cdot 2^{1+\kappa}\tau\e_{l+1}^{1+\kappa},
\end{align*}
which means that %is equivalent to
\begin{align*}
B_{2^{1+\kappa}\tau\e_{l+1}^{1+\kappa}}(\e_{l+1} z_i)\cap B_{2^{1+\kappa}\tau\e_{l+1}^{1+\kappa}}(\e_{l+1} z_j)\neq\emptyset.
\end{align*}
The domain $D$ being star-shaped implies monotonicity of $\Phi^{\eps}(D)$ in $\eps$, in particular $\Phi^\e(D)\subset\Phi^{\e_{l+1}}(D)$, which combined with the previous display yields
%This yields together with \eqref{eq:PhiSubsetPhi_l}
\begin{align*}
A_\tau^\e\subset A_{\bar{\tau}}^{\e_{l+1}},
\end{align*}
thus showing that \eqref{eq:BorCant2} implies \eqref{eq:BorCant1}.

It remains to show~\eqref{eq:BorCant2}. Let $\eps > 0$ and $\tau \ge 1$ be fixed. Observe that if for $z_i,z_j\in\Phi^\e(D)$ we have $B_{\tau\e^{1+\kappa}}(\e z_i)\cap B_{\tau\e^{1+\kappa}}(\e z_j)\neq\emptyset$, then $\e|z_i-z_j|\leq 2\tau\e^{1+\kappa}$ and after simplifying $|z_i-z_j|\leq 2\tau\e^\kappa$. In other words
\begin{align}\label{eq:Asubset}
A_\tau^\e\subset \{\omega \in \Omega : \exists x\in\frac1\e D:\#(\Phi^\e(D)\cap B_{2\tau \e^\kappa}(x))\geq 2\}.
\end{align}

Recall that for $S\subset\R^d$, we denote by $N(S)=\#(S\cap\Phi)$ the random variable providing the number of points of the process which lie inside $S$. Let us also note that the points are distributed according to a Poisson distribution %point process in $\e^{-1}D$ 
with intensity $\lambda>0$. We now recall a basic estimate from~\cite[Proof of Lemma 6.1]{GiuntiHoefer19}%\footnote{There is not Lemma 6.1 in that paper, or?}
: For $0<\eta<1$, define the set of cubes with side length $\eta$ centred at the grid $\eta\Z^d$ by
\begin{align*}
\mathcal{Q}_\eta:=\{y+[-\eta/2,\eta/2]^d:y\in\eta\Z^d\}.
\end{align*}
Since it is not true that any ball of radius $\frac \eta 4$ is contained in one of these cubes, we need to add (finitely many) shifted copies of $\mathcal{Q}$. For that let $S_\eta$ be the vertices of the cube $[0,\eta/2]^d$, i.e.,
\begin{align*}
S_\eta=\{z=(z_1,\dots,z_d)\in\R^d: z_k\in \{0,\eta/2\} \text{ for } k=1,\dots,d\}.
\end{align*}
Observe that for any $x\in\R^d$, there exist $z\in S_\eta$ and a cube $Q\in \mathcal{Q}_\eta$ such that $B_\frac{\eta}{4}(x)\subset z+Q$, which immediately implies
\begin{align*}
\P(\exists x\in\frac1\e D&: N(B_\frac{\eta}{4}(x))\geq 2)\\
&\le \P(\exists Q\in\mathcal{Q}_\eta, z\in S_\eta: (z+Q)\cap\frac1\e D\neq\emptyset, N(z+Q)\geq 2).
\end{align*}
Since $S_\eta$ has $2^d$ elements and the number of cubes $Q\in \mathcal{Q}_\eta$ that intersect $\e^{-1} D$ is bounded by $C(D)(\e\eta)^{-d}$, we use the distribution of Poisson point process to conclude
\begin{align*}
\P(\exists x\in\frac1\e D: N(B_\frac{\eta}{4}(x))\geq 2) &\le \sum_{z \in S_\eta} \sum_{ Q } \P( N(z+Q) \ge 2) 
\\
&\le 
2^d C(D)(\e \eta)^{-d}e^{-\lambda \eta^d}\sum_{k=2}^\infty \frac{(\lambda\eta^d)^k}{k!}
\\
&\leq C(D) 2^d(\e\eta)^{-d}(\lambda\eta^d)^2,
\end{align*}
where the last inequality follows from~\eqref{eq:exp}. Letting $\eta_\e:=8\tau\e^\kappa$, this together with \eqref{eq:Asubset} and the fact that $\#(\Phi^\e(D)\cap S)\leq N(S)$ for any $S\subset\R^d$, yields
\begin{align*}
\P(A_\tau^\e)\leq C(\e^{1+\kappa})^{-d} \e^{2d\kappa}=C\e^{d(\kappa-1)}.
\end{align*}
To show~\eqref{eq:BorCant2} we take a sum over $l$ with $\e=\e_l = 2^{-l}$, which using $\kappa>1$ can be estimated as
\begin{align*}
\sum_{l=0}^\infty \P(A_{\bar{\tau}}^{\e_l})\leq C\sum_{l=0}^\infty 2^{-l d (\kappa-1)}<\infty,
\end{align*}
and~\eqref{eq:BorCant2} follows from direct application of the Borel-Cantelli lemma.

\medskip\noindent
{\bf Part (2): } 
We now turn to the second part of the theorem, i.e., the construction of boxes $I^\e_i$. Fixing $\eps$, the first step is to construct a finite collection $\I = \{ \tilde I_i \}$ of auxiliary boxes such that:
\begin{itemize}
 \item these boxes cover the points, i.e., $\bigcup_{i} \tilde I_i \supset \eps\Phi^{\eps}(D)$,
 \item $\dist_{\infty}(\tilde I_i,\tilde I_j) \ge \frac1{2N} \eps^{1+\delta}$ ,
 \item $s(\tilde I_i) \le \eps^{1+\delta}$, where $s(I)$ of a box $I$ denotes the size of its longest side,
 \item each box $\tilde I_i$ satisfies $|\tilde I_i \cap \eps \Phi^\e(D)| \le N$.
\end{itemize}
Here the crucial condition is the second one, i.e., that the boxes are well-separated. 
% The idea of the proof is to cover the domain $D$ with cubes of size $\e^{1+\delta}$, split these cubes into even smaller ones and from these small cubes build the boxes $I^\e$ by growing them layer by layer. Before we start, recall the definition of $C_\e$ in Theorem \ref{thm:MainProbThm} \ref{lem:finitemanyballs} as
% \begin{align*}
% C_\e=x+\frac12 \e^{1+\delta} [-1,1]^3,
% \end{align*}
% 
% where $\delta=(\alpha-3)/3>0$. For legibility, we will not write the dependence on $\e$ when no ambiguity occurs.\\
% 
% First, we split each $C_\e$ into $(2N)^3$ cubes of side length $l:=(2N)^{-1}\e^{1+\delta}$ and call them ``small cubes''. Further, if $I$ is a box as constructed in a minute, we will not distinguish between the index set $\I$ and the set of boxes $\{I_\nu\}_{\nu\in\I}$. We will show that we can find a collection $\I$ of ``inner'' boxes $I_i$ such that
% \begin{itemize}
%  \item $\dist(I_i,I_j) \ge l$ for all $i\neq j$, %(this condition replaces the role of outer boxes)
%  \item $s(I_i)\! \le\! Nl =\eps^{1+\delta}\!/2$, where $s(I)$ of a box $I$ denotes the size of its longest side (called the \emph{length} of $I$),
%  \item each box $I_i$ satisfies $\#(I_i \cap \eps \Phi^\e(D)) \le N$.
% \end{itemize}
%
Let $l = \frac1{2N}\eps^{1+\delta}$. 
% $\mathcal{C_\eps}$ denote the set of all cubes of the form $x + \frac l 2 (-1,1]^3$ for all $x \in l \mathbb{Z}^3$, i.e. we decompose $\mathbb{R}^3$ into cubes with side length $l$. 
% \medskip
We will grow the boxes $\tilde I$ from the collection $\I$ step by step, starting with cubes of side length $l$. At every moment of this growth process, every box $\tilde I \in \I$ will satisfy the following conditions:
{
\renewcommand{\theenumi}{\roman{enumi}}
\begin{enumerate}
 \item $\tilde I = [a_xl,b_xl) \times [a_yl,b_yl) \times [a_zl,b_zl)$ for some $a_x,b_x,a_y,b_y,a_z,b_z \in \mathbb{Z}$, i.e., each box is a union of many small cubes;\label{first}
 \item for each $a \in [a_x,b_x) \cap \mathbb{Z}$ holds $[al,(a+1)l) \times [a_yl,b_yl) \times [a_zl,b_zl) \cap \eps \Phi^\e(D) \neq \emptyset$, and similarly for $y$ and $z$, i.e., in every slice there is some point from $\eps \Phi^\e(D)$;  \label{second}
 \item $\#(\tilde I \cap \eps\Phi^\e(D)) \le N$. \label{third}
\end{enumerate}
}

\medskip
At the beginning, let $\I$ consist of all cubes $[a_xl,(a_x+1)l)\times[a_yl,(a_y+1)l)\times[a_zl,(a_z+1)l)$ which have a point from $\eps \Phi^\e(D)$ in it. 
Since $D$ is bounded, $\I$ consists of finitely many boxes (cubes). We then repeat the following procedure:

\smallskip
If there exist two different boxes $\tilde I, \tilde J \in \I$ such that $\dist(\tilde I,\tilde J) = 0$, we fix them and merge them together. That means, we remove $\tilde I = [a_xl,b_xl) \times [a_yl,b_yl) \times [a_zl,b_zl)$ and $\tilde J = [a_x'l,b_x'l) \times [a_y'l,b_y'l) \times [a_z'l,b_z'l)$ from $\I$ and add 
\begin{align*}
\tilde K &= [A_xl,B_xl) \times [A_yl,B_yl) \times [A_zl,B_zl)\\
:\!\!&= [(a_x \wedge a_x')l,(b_x \vee b_x')l) \times [(a_y  \wedge a_y')l,(b_y  \vee b_y')l) \times [(a_z  \wedge a_z')l,(b_z  \vee b_z')l)
\end{align*}
to $\I$ instead. Here $\wedge$ and $\vee$ stand as usual for minimum and maximum, respectively. 

First, observe that (\ref{first}) trivially follows from the definition of $\tilde K$. Next, to verify that $\tilde K$ satisfies (\ref{second}), let us fix $i \in \{ x,y,z \}$, and observe that $\dist(\tilde I,\tilde J) = 0$ implies $[a_i,b_i] \cap [a_i',b_i'] \neq \emptyset$. Hence, for any ${a \in [\min(a_i,a_i'),\max(b_i,b_i'))}$ either $a \in [a_i,b_i)$, in which case (\ref{second}) for $\tilde I$ implies (\ref{second}) for $\tilde K$, or $a \in [a_i',b_i')$, in which case (\ref{second}) for $\tilde J$ implies (\ref{second}) for $\tilde K$. 

It remains to argue that $\tilde K$ satisfies also (\ref{third}). Since $\tilde I$ satisfies both (\ref{second}) and (\ref{third}), in particular to each $a \in [a_i,b_i)\cap \Z$ there is assigned at least one point from $\eps \Phi^\e(D)$ and there are at most $N$ such points, it follows that $[a_i,b_i)$ has length at most $Nl$. The same argument applies verbatim to $\tilde J$, and so the union of $[a_i,b_i)$ and $[a_i',b_i')$ has length at most $2Nl$. Hence, each side of $\tilde K$ has length at most $s(\tilde K)\leq 2N  l = 2N \frac1{2N} \eps^{1+\delta} = \eps^{1+\delta}$. In addition $\tilde K$ satisfies (\ref{first}), and so there exists a (closed) cube $Q_{\tilde K}$ with side length $\eps^{1+\delta}$ such that $\tilde K \subset Q_{\tilde K}$. By Theorem \ref{thm:MainProbThm}, the number of points in $Q_{\tilde K}$ is at most $N$, which implies the same for $\tilde K$, i.e.,
\begin{align*}
\#(\tilde K \cap \eps \Phi^\e(D)) \le \#(Q_{\tilde K} \cap \eps \Phi^\e(D)) \le N,
\end{align*} 
which shows (\ref{third}) for $\tilde K$; moreover, since $\tilde K$ also fulfils (\ref{second}), this also shows that $\tilde K$ has length at most $s(\tilde K)\leq Nl$.

Since the collection $\I$ was finite at the beginning, and in each iteration we decrease the number of boxes in $\I$ by one (we remove $\tilde I$ and $\tilde J$ and add $\tilde K$), this process has to terminate. In particular, at the end $\I$ consists of boxes which have positive distance from each other, since otherwise the process would not terminate at this point. Since all boxes in $\I$ satisfy (\ref{first}), this in particular implies that this positive distance has to be at least $l = \frac1{2N} \eps^{1+\delta}$. Moreover, since each box has side length at most $Nl = \frac{1}{2} \eps^{1+\delta}$, and each point in $\e\Phi^{\eps}(D)$ is at least distance $\eps$ to $\partial D$, we see that each box (and actually also its small neighborhood) lies inside $D$.

Using boxes from $\I$ we define boxes $I^\eps_i$: for each auxiliary box $\tilde I_i \in \I$ set $I^\eps_i := \{ x \in \R^3 : \dist_\infty(x,\tilde I_i) \le \frac{1}{8N} \eps^{1+\delta} \}$, and it remains to show that $\{ I^\eps_i \}$ satisfy~\eqref{thm31labela}-\eqref{thm31labele}. First, by the assumption $\kappa > \delta$, and so for small enough $\eps$ we have $\eps^{1+\kappa} \le \frac{1}{16N} \eps^{1+\delta}$. Therefore, by the triangle inequality we have for any $\eps z \in \tilde I_i$ that $\dist_{\infty} (B_{\eps^{1+\kappa}}(\eps z),\partial I^\eps_i) \ge \frac{1}{8N}\eps^{1+\delta} - \eps^{1+\kappa} \ge \frac{1}{16N}\eps^{1+\delta}$, thus~\eqref{thm31labela} and~\eqref{thm31labelc} hold. Since by the construction the auxiliary boxes satisfy $\dist_\infty(\tilde I_i,\tilde I_j) \ge \frac{1}{2N}\eps^{1+\delta}$, and all the points from $\eps \Phi^{\eps}(D)$ are inside these boxes, we see that $I^\eps_i \setminus \tilde I_i$ contains no point from $\eps \Phi^\eps(D)$. Therefore \eqref{third} for $\tilde I_i \in \I$ implies~\eqref{thm31labelb} for $I^\eps_i$. Finally, \eqref{thm31labeld} trivially follows from the definition of $I^\eps_i$ and the separation of elements in $\I$ in form of $\dist_\infty(\tilde I_i,\tilde I_j) \ge \frac{1}{2N} \eps^{1+\delta}$, and \eqref{thm31labele} uses that $I^\eps_i$ consist in each direction of at least one cube and of at most $N$ of them.

\end{proof}

\section{Proofs of Lemma~\ref{lem:O-EpsJohn} and Theorem~\ref{thm:MainBog}}\label{sec:BogOp}

Before proving that a box from which we remove finitely many small well-separated balls is a John domain, let us recall how John domains are defined:\newline
%First, let us recall the definition of John domain as stated in Definition~\ref{def:John}: 
For a constant $c>0$, a domain $U\subset \R^d$ is said to be a $c-$John domain if there exists a point $x_0\in U$ such that for any point $x\in U$ there is a rectifiable path $\Gamma:[0,\ell ]\to U$ which is parametrized by arc length with
\begin{align}\label{eq:JohnProp}
\Gamma(0)=x,\quad \Gamma(\ell)=x_0,\quad \forall t\in [0,\ell]: |\Gamma(t)-x|\leq c\, \dist(\Gamma(t),\d U).
\end{align}

John domains may have fractal boundaries or internal cusps, whereas external cusps are forbidden. For instance, the interior of Koch's snowflake as well as any convex domain are John domains. In the case of bounded domains, there are several equivalent definitions of John domains, see \cite[Section 2.17]{VaisalaJohn}. We state the following characterisation, which is used in \cite[Section 3.1]{Ruzicka}: A bounded domain $U$ is a $c-$John domain in the sense of Definition \ref{def:John} if and only if there is a $c_1(c)>0$ and a point $x_0\in U$ such that any point $x\in U$ can be connected to $x_0$ by a rectifiable path $\Gamma:[0,\ell ]\to U$ which is parametrized by arc length and
\begin{align*}
\bigcup\limits_{t\in [0,\ell]} B\big(\Gamma(t), t/c_1\big)\subset U.
\end{align*}

One way how to prove Lemma~\ref{lem:O-EpsJohn} is inductively by showing, that under some assumption on a ball one can remove it from a John domain while changing the John constant of the domain by at most a fixed factor -- for that we would need to modify arcs which run close to (or through) this removed ball while estimating how much does this change the situation. For a similar argument with small balls replaced with points, see~\cite[Theorem 1.4]{Huang+08}. Assuming this, since we have to remove at most $N$ balls and at the beginning the domain is rectangle with proportional sides, in particular a John domain, this would lead to the conclusion. 

Instead of this we provide a direct constructing argument:

\begin{proof}[Proof of Lemma~\ref{lem:O-EpsJohn}.]
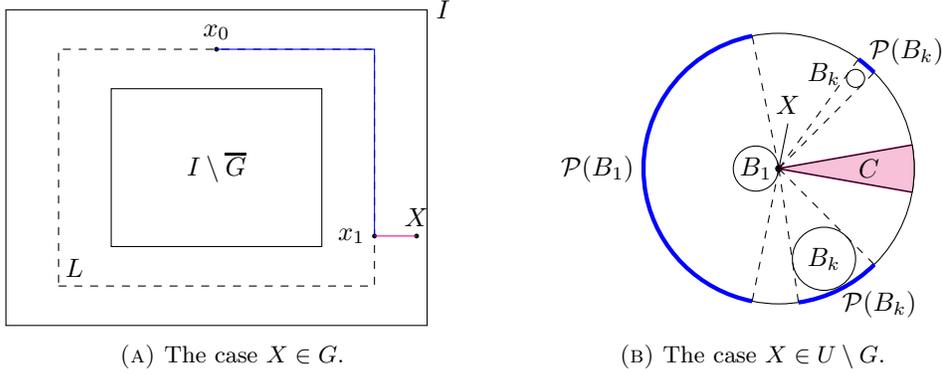
\begin{figure}
\begin{subfigure}[b]{0.45\linewidth}
\centering
\begin{tikzpicture}[scale=0.7]
\draw (-4,-3) rectangle (4,3) node[anchor=west] {$I$};
\draw (-2,-1.5) rectangle(2,1.5);
\node at (0,0) {$I\setminus \overline{G}$};
\draw[dashed] (-3, -2.25) rectangle (3,2.25);
\node at (-2.7,-1.9) {$L$};
\filldraw (3.8,-1.3) circle (1pt) node[anchor=south] {$X$};
\filldraw (3,-1.3) circle (1pt) node[anchor=east] {$x_1$};
\draw[magenta] (3.8,-1.3) -- (3,-1.3);
\filldraw (0,2.25) circle (1pt) node[above] {$x_0$};
\draw[blue] (3,-1.3) -- (3,2.25) -- (0,2.25);
\end{tikzpicture}
\subcaption{The case $X\in G$.\label{fig:XinG}}
\end{subfigure}
\begin{subfigure}[b]{0.45\linewidth}
\centering
\begin{tikzpicture}[scale=0.6]%PB: changed from 0.7 to 0.6
\tkzDefPoint(0,0){X};
\filldraw (X) circle (2pt);
\node at (.2,1) [above]{$X$};
\draw (X)--(.2,1);
\draw (X) circle (3cm);
\draw (-0.51,0) circle (.5cm);
\node at (-.51,0) {$B_1$};
\draw[dashed] (X)--(101.36488:3cm);
\draw[dashed] (X)--(-101.36488:3cm);
\node at (-3,0) [anchor=east] {$\mathcal{P}(B_1)$};
\draw[ultra thick, blue] (101.36488:3cm) arc (101.36488:258.63512:3cm);
\draw (1.7,2) circle (0.2cm);
\node at (1,2.1) {$B_k$};
\draw[dashed] (X)--(45.25654:3cm);
\draw[dashed] (X)--(54.00529:3cm);
\node at (1.8,2.6) [anchor=west] {$\mathcal{P}(B_k)$};
\draw[ultra thick, blue] (45.25654:3cm) arc (45.25654:54.00529:3cm);
\draw (1,-2) circle (0.7cm);
\node at (1,-2) {$B_k$};
\draw[dashed] (X)--(278.32195:3cm);
\draw[dashed] (X)--(314.80816:3cm);
\node at (1.2,-3) [anchor=west] {$\mathcal{P}(B_k)$};
\draw[ultra thick, blue] (278.32195:3cm) arc (278.32195:314.80816:3cm);
\draw[magenta, thick] (X)--(10:3cm);
\draw[magenta, thick] (X)--(-10:3cm);
\filldraw[fill=magenta!30] (X)--(-10:3cm) arc (-10:10:3cm)--cycle;
\node at (2,0) {$C$};
\end{tikzpicture}
\subcaption{The case $X\in U\setminus G$.\label{fig:Cone}}
\end{subfigure}
\caption{(A) The point $X\in G$, first connected to $x_1\in L$ (red) and then to $x_0$ while not leaving $L$ (blue). (B) The projections (blue) of the balls $B_1$ and $B_k$ onto the sphere $S$ with midpoint $X$. The cone $C$ illustrated by the red area hits none of the balls and serves as the ``outgoing'' sector from $X$ to $L$.\label{fig:figure}}
\end{figure}
To start, we use Theorem~\ref{thm:MainProbThm}, part~\eqref{veta31jedna}, twice: once with {\color{red} $\kappa = \kappa_1 := 1+\delta$ and second time with $\kappa = \kappa_2 := \alpha-1-\frac3m-\delta$. Observe that both values of $\kappa$ are within the admissible range $(\max(1,\delta),\alpha-1-\frac3m)$,} and therefore the theorem yields the following:
\noindent
there exists an a.s.~positive $\eps_0(\omega)$, obtained as the smaller of the two $\eps_0$, such that for $0 < \eps \le \eps_0$ holds:
\begin{equation}\label{eq34.1}
 \max_{z_j \in \Phi^{\eps}(D)} \eps^\alpha r_j \le \eps^{1+\kappa_2} \quad \textrm{and} \qquad |\eps z_j - \eps z_k| \ge 2\eps^{1+\kappa_1} \quad \textrm{ for any } z_j, z_k \in \Phi^\eps(D). 
\end{equation}

\providecommand{\w}{\frac{1}{32N}\eps^{1+\delta}}

Assume we have $0 < \eps \le \eps_0$ small enough and recall that we want to show that
\begin{equation*}
U:=I^\eps_i \setminus \bigcup_k B_k %\limits_{z_j\in I_i^\e} B_{\e^\alpha r_j}(\e z_j)
\end{equation*}
is a $c(N)-$John domain in the sense of Definition~\ref{def:John}, where ${\{ B_k \}_k = \{ B_{\e^\alpha r_k}(\e z_k) : \e z_k \in I^{\eps}_i \}}$. For brevity we set $I := I^{\eps}_i = p + (-l_1/2,l_1/2)\times(-l_2/2,l_2/2)\times(-l_3/2,l_3/2) \subset \R^3$, where $p$ is the center and $l_i$ are the side lengths of $I$. Since \eqref{eq:JohnProp} is scale-invariant,we can assume $l_1 \ge l_2 \ge l_3 = 1$. The set
\begin{equation*}
L := \{ x \in I : \dist_\infty(x,\partial I) = \w \} % 2*w will denote the width of the bdry layer
\end{equation*}
will serve as a ``highway'' in the set $U$, and for the specific point $x_0$ from Definition~\ref{def:John} we choose $x_0 := p + (0,0,l_3/2-\w)$. We also denote the ring around $L$ by $G := \{ x \in I : \dist_{\infty}(x,L) < \w \}$. 

To show that $U$ is a John domain, for each $X \in U$ we need to construct a path from $X$ to $x_0$ along which $|\cdot - X| \le c \dist(\cdot,\partial U)$. The idea is first to go from $X$ to $L$, and then run along $L$ to $x_0$. Observe that for points $x \in L$ the condition is easy to satisfy: for each $x \in L$ we have $\dist(x,\partial U) = \dist(x,\partial G) = \w$ and $|x-X| \le \diam(U) \le \sqrt{3}\, l_1$, and so using~\eqref{thm31labele} to see $l_1 \le C(N)\eps^{1+\delta}$ we get that $|x-X| \le c(N) \dist(x,\partial U)$ as required. 

It remains to describe the path from $X$ to $L$. For points $X \in G$ this is straightforward (see Figure \ref{fig:XinG}): we just choose the shortest path from $X$ to $L$ and observe that any point $x$ on that path satisfies
$\dist(x,\partial U) \ge \dist(x,\partial G) \ge 3^{-1/2} |X-x|$. The $\sqrt{3}$ is optimal as can be seen from points in corners. 

%We start with points near the highway $L$. Let $G := \{ x \in \R^3 : \dist_{\infty}(x,L) < w \}$, and observe that $\dist(B_{\eps^\alpha r_j}(\eps z_j),\partial I^\e_i) > 2w$ for all $z_j \in I^\eps_i$ implies $G \subset U$. For $x \in G$, we denote $x_1 \in L$ the closest point from $L$ to $x$. We denote the shortest path from $x$ to $x_1$ by $\Gamma_1$, and the shortest polygonal chain from $x_1$ to $x_0$ by $\Gamma_2$. Since $G \subset U$, for $y \in \Gamma_1$ holds $\dist(y,\partial U) \ge \dist(y,\partial G) \ge 3^{-1/2} |x-y|$. The $\sqrt{3}$ is optimal and as can be seen from points in corners. For $y \in \Gamma_2$ it is enough to use $\dist(y,\partial U) = \dist(y,\partial G) = w$. Indeed, for any $y \in \Gamma_2$ we have $|x-y| \le \sqrt{3}l_1$ as well as $l_1 \le C(N)$ (see~\ref{thm31labele}), which combined with the previous fact yields along $\Gamma_2$ that $\dist(y,\partial U) = w \ge C^{-1} l_1 \ge C |x-y|$. 

In the rest of the proof we deal with the points from the ``interior'' $U \setminus G$. For $X \in U \setminus G$ we need to construct a path from $X$ to $L$, while not going too close near the balls $B_k$. We will use two important properties of these balls: the size of the balls is much smaller than their mutual distance (see~\eqref{eq34.1}), and there are at most $N$ of them. We fix $X \in U\setminus G$ and show that we can actually use a straight line to connect $X$ with $L$. Along this line we should be able to move a growing ball without hitting $\{B_k\}$, what is equivalent to an existence of a cone with an opening $c(N)$ which avoids all the balls. For that let $S$ be a unit sphere centered at $X$, and let $\mathcal P$ denote the orthogonal projection on $S$. Farther, let $P := \mathcal P(\bigcup_k B_k)$ denote the projection of balls on $S$. Observe that if we find a disc on $S$ of fixed radius (depending on $N$) which does not overlap with $P$, then we are done since such disc corresponds to a cone at $X$ avoiding the balls $B_k$ (see Figure \ref{fig:Cone}). 

Hence, we reduced our task to a problem of finding a not too small disc in $S \setminus P$, with $P$ being a union of at most $N$ discs with some additional properties. First, it can happen that $X$ lies very close to one of the balls, so that the projection of this particular ball on $S$ covers (almost) half of the sphere $S$. For this reason w.l.o.g.~let $B_1$ denote the ball whose center is closest to $X$, which we treat separately: let $S' \subset S$ be a half-sphere with the pole being exactly opposite to the center of $\mathcal P(B_1)$, in particular $\mathcal P(B_1)$ and $S'$ are disjoint. Since $B_1$ was the closest ball to $X$, it follows from the second estimate in~\eqref{eq34.1} that $X$ is at least $\eps^{1+\kappa_1}$ away from centers of the remaining balls $B_k, k \ge 2$. 
On the other hand, the first relation in~\eqref{eq34.1} bounds the radii of these balls with $\eps^{1+\kappa_2}$. Therefore, the projections of these remaining balls are discs of radius at most $C\frac{\eps^{1+\kappa_2}}{\eps^{1+\kappa_1}} = C\eps^{\kappa_2 - \kappa_1}$. {\color{red}Since $\kappa_2 - \kappa_1 = \alpha-2-\frac3m-2\delta>0$ by the choice of $\delta$ in \eqref{eq:DefDelta},} we see that for $\eps$ small these (at most $N-1$) projections are tiny discs (almost points). We can now find a radius $r=r(N)$ with the following property: there exist $N$ discs $D_1,\ldots,D_N$ of radius $r$ in $S'$ such that the distance between any two discs is at least $r$ as well. One option is to arrange them along the boundary of $S'$ with necessary spacing between them, thus achieving $r \sim N^{-1}$. Provided now $\eps$ is small enough so that the radii of $\mathcal P(B_k)$, which are bounded by $C \eps^{\kappa_2 - \kappa_1}$, are smaller than $r$, we are done: there are $N$ discs $D_1,\ldots,D_N$ and at most $N-1$ projections $\mathcal P(B_k)$ where each projection can touch at most one $D_k$, so that one disc will not overlap with any of the projections $\mathcal P(B_k)$, thus defining the cone we are searching for.

This solution to the last question is naturally far from optimal (in $r$): consider a well-studied question of finding an optimal cover of a sphere (more precisely half of it) with $N$ identical discs of smallest radius. If $\rho$ denotes the smallest such radius, then for any configuration of $N$ points in $S'$ there exists a disc in $S'$ of radius $\rho$ which avoids them, thus also providing a solution to our problem \cite{Toth1949}.

\end{proof}

Since the perforated boxes $U$ from Lemma~\ref{lem:O-EpsJohn} are uniform John domains, in particular we have a Bogovki\u{\i} operator on each $U$, Theorem~\ref{thm:MainBog} can be proved along the lines of the proof in~\cite{DieningFeireislLu}. First, using a Bogovski\u{\i} operator on the whole $D$ we obtain a function $\vb u$ with the correct divergence, but which naturally does not vanish on the holes. To achieve that, we modify $\vb u$ in each box $I^{\eps}_i$. More precisely, near $\partial I^\eps_i$ in a boundary layer of size $\frac{1}{16N}\eps^{1+\delta}$ we change $\vb u$ to its average value over this layer, and then inside the box (where also the balls are removed) cut off this constant function near each hole over a scale $\eps^\alpha$. Since by this modification we also change the divergence of the function, we employ Bogovski\u{\i}'s operator both on each box as well as near each hole to fix the divergence. 

% \subsection{Proof of Theorem \ref{thm:MainBog}}
\begin{proof}[Proof of Theorem~\ref{thm:MainBog}]

Let us recall definition of $D_\e=D\setminus\bigcup_{z_j\in\Phi^\e(D)} B_{\e^\alpha r_j}(\e z_j)$ (see~\eqref{def:Domain}). 
To prove the theorem we construct a linear operator of Bogovski\u{\i} type, bounded independently of $\eps$:
\begin{align*}
\B_\e:L_0^q(D_\e)\to W_0^{1,q}(D_\e)
\end{align*}
satisfying
\begin{align}\label{eq:BogIt2}
\div\B_\e(f)=f \text{ in } D_\e,\quad ||\B_\e(f)||_{W_0^{1,q}(D_\e)}\leq C\, ||f||_{L_0^q(D_\e)}.
\end{align}

For $1<q<\infty$ and $f\in L_0^q(D_\e)$ we denote by $\tilde{f}\in L_0^q(D)$ its zero extension in the holes. Using classical Bogovski\u{\i}'s operator in Lipschitz domain $D$~\cite[Chapter 3]{Galdi2011}, norm of which depends on the Lipschitz character of $D$, we can find a function $\vb u=\B_D(\tilde{f})\in W_0^{1,q}(D)$ satisfying
\begin{align*}
\div \vb u =\tilde{f}\text{ in } D,\quad \|\vb u\|_{W_0^{1,q}(D)}\leq C\, \|\tilde{f}\|_{L_0^q(D)}=C\, \|f\|_{L_0^q(D_\e)}
\end{align*}
with $C=C(D,q)$. 
% We recall the definition of $O^\e$ and $I^\e$ from Theorem \ref{thm:MainProbThm} \ref{lem:GroupingBalls} to write
% \begin{align*}
% A^\e:=O^\e\setminus \overline{I^\e}.
% \end{align*}

\providecommand{\ii}{I^{\eps,\textrm{in}}_i}

{\color{red}Since $\alpha-3/m>2$, by applying} Theorem~\ref{thm:MainProbThm} we find for every $\eps > 0$ small enough a finite collection of boxes $I^{\eps}_i$ such that for any point $z_j \in \Phi^\eps(D)$ there is $i$ such that 
\begin{equation*}
B_{\e^{\alpha} r_j }(\e z_j)\subset B_{2 \e^{\alpha} r_j}(\e z_j)\subset B_{\eps^{1+\kappa}}(\e z_j) \subset
I^{\e,\textrm{in}}_i,
\end{equation*}
where
\begin{align*}
I^{\eps,\textrm{in}}_i &:= \{ x \in I^{\e}_i : \dist_\infty(x,\partial I^\eps_i) \ge \frac{1}{16N}\eps^{1+\delta} \}.
\end{align*}
For any box $I^\eps_i$ and any ball $B_{\eps^\alpha r_j}(\e z_j)$ consider the corresponding cut-off functions:
%consider two cut-off functions $\chi_\e$ and $\zeta_{\e,i}$ which satisfy
\begin{align}
&\chi_{\eps,i} \in C_c^\infty (I^\eps_i),\quad \chi_{\e,i} \restriction_{\ii}=1,\quad \|\nabla \chi_{\e,i}\|_{L^\infty(D)}\lesssim \e^{-(1+\delta)},\label{def:chi}\\
&\zeta_{\e,j}\in C_c^\infty \Bigl(B_{2 \e^\alpha r_j}(\e z_j)\Bigr),\quad \zeta_{\e,j}\restriction_{B_{\e^\alpha r_j}(\e z_j)}=1,\quad {\color{red}\|\nabla \zeta_{\e,j}\|_{L^\infty(B_{2\e^\alpha r_j}(\e z_j))}\lesssim \frac{1}{r_j}\e^{-\alpha}.}\label{def:zeta}
\end{align}%
Further we denote the mean value of $\vb u$ over a measurable set $S\subset\R^3$ by
\begin{align*}
\langle \vb u\rangle_S:=\frac{1}{|S|}\int_S \vb u
\end{align*}
and define
\begin{align*}
A^{\eps}_i &:= I^{\eps}_i \setminus I^{\eps,\textrm{in}}_i = \{ x \in I^{\eps}_i : \dist_\infty(x,\partial I^{\eps}_i) < \frac{1}{16N}\eps^{1+\delta} \},\\
\vb b_{\e,i}(\vb u)&:=\chi_{\e,i} (\vb u-\langle \vb u\rangle_{A^\e_i})\in W_0^{1,q} (I^\eps_i),\\
\beta_{\e,j}(\vb u)&:=\zeta_{\e,j}\, \langle \vb u\rangle_{A^\e_i}\in W_0^{1,q} \Bigl(B_{2 \e^\alpha r_j}(\e z_j)\Bigr),
\end{align*}
where as before $i$ and $j$ are related through $\eps z_j \in I^\e_i$. 

Since all the lengths in the set $A^\e_i$ are proportional to $\eps^{1+\delta}$ (with the proportionality depending on $N$), Poincar\'e's inequality implies 
\begin{align*}
\|\vb u-\langle \vb u\rangle_{A^\e_i}\|_{L^q(A^\e_i)} \lesssim \e^{1+\delta}\, \|\nabla \vb u\|_{L^q(A^\e_i)},
\end{align*}
and by \eqref{def:chi} we get
\begin{align}\label{eq:estimateB(u)}
\begin{split}
\|\nabla \vb b_{\e,i}(\vb u)\|_{L^q(A^\e_i)}&\leq \|\chi_{\e,i} \nabla (\vb u-\langle \vb u\rangle_{A_i^\e})\|_{L^q(A_i^\e)}+\|\nabla \chi_\e (\vb u-\langle \vb u\rangle_{A_i^\e})\|_{L^q(A_i^\e)}\\
&\lesssim \|\nabla (\vb u-\langle \vb u\rangle_{A^\e_i})\|_{L^q(A^\e_i)}+\e^{-(1+\delta)}\, \|\vb u-\langle \vb u\rangle_{A^\e_i}\|_{L^q(A^\e_i)}\\
&\lesssim \|\nabla \vb u\|_{L^q(A^\e_i)}.
\end{split}
\end{align}
{\color{red}Similarly, by \eqref{def:zeta} and Jensen's inequality, we get
\begin{align}\label{eq:estimateBeta(u)}
\begin{split}
\|\nabla \beta_{\e,j}(\vb u)\|_{L^q(B_{2 \e^\alpha r_j}(\e z_j))} &=\|\nabla\zeta_{\e,j}\cdot \langle \vb u\rangle_{A^\e_i}\|_{L^q(B_{2 \e^\alpha r_j}(\e z_j))}\\
&\lesssim r_j^{\frac3q-1}\e^{\big(\frac3q-1\big)\alpha}\, |\langle \vb u\rangle_{A^\e_i}| \lesssim r_j^{\frac3q-1}\e^{\big(\frac3q-1\big)\alpha}\, |A^\e_i|^{-\frac1q} \|\vb u\|_{L^q(A^\e_i)}\\
&\lesssim r_j^{\frac3q-1}\e^{\big(\frac3q-1\big)\alpha-\frac{3(1+\delta)}{q}}\, \|\vb u\|_{L^q(A^\e_i)}.
\end{split}
\end{align}
%
%Recall the definitions of $m$, $\delta$ and $\kappa_2$ as
%\begin{align*}
%m>\frac{3}{\alpha-3},\quad \delta=\frac{\alpha-3(1+\frac1m)}{4}>0,\quad \kappa_2=\alpha-1-\frac3m-\delta.
%\end{align*}
%
Since $B_{\e^\alpha r_j}(\e z_j)\subset D$, we have $r_j\leq \e^{1+\kappa_2-\alpha}=\e^{-(\frac3m+\delta)}$ by \eqref{eq34.1} and the choice of $\kappa_2=\alpha-1-\frac3m-\delta$. This yields
\begin{align*}
r_j^{\frac3q-1}\e^{\big(\frac3q-1\big)\alpha-\frac{3(1+\delta)}{q}}\leq \e^{(\frac3q-1)(\alpha-\frac3m-\delta)-\frac3q (1+\delta)}.
\end{align*}
Thus, by choosing $\delta$ such that
\begin{align}\label{DeltaJedna}
\delta\leq\frac{(3-q)(\alpha-\frac3m)-3}{6-q},
\end{align}
we get $\delta>0$ and uniform bounds on $\|\beta(\vb u)\|_{L^q(B_{2\e^\alpha r_j}(\e z_j))}$ for all $1<q<3$ which satisfy \eqref{Conditionm}.}

Since $\beta_{\e,i}$ as well as $\vb b_{\e,j}$ do not have vanishing divergence, we need to correct them using Bogovki\u{\i} operators on perforations of $B_{2\eps^\alpha r_j}(\eps z_j)$ and $I^\e_i$, respectively. In the first case one can construct the Bogovki\u{\i} operator
\begin{equation}
\widetilde \B_{\e,j}:L_0^q\big(B_{2 \e^\alpha r_j}(\e z_j)\setminus B_{\e^\alpha r_j}(\e z_j)\big) \to W_0^{1,q}\big(B_{2 \e^\alpha r_j}(\e z_j)\setminus B_{\e^\alpha r_j}(\e z_j)\big) 
\end{equation}
as in \cite[Chapter III and Theorem III.3.1]{Galdi2011}, which mimics the original proof from Bogovski\u{\i} in \cite{Bogovskii80}. Alternatively, one can also construct it by observing that $B_{2 \e^\alpha r_j}(\e z_j)\setminus B_{\e^\alpha r_j}(\e z_j)$ is a uniform John domain independent of $\e$ and $z_j$ and use \cite[Theorem 3.8 and Theorem 5.2]{Ruzicka}. In the second situation, the existence of the Bogovski\u{\i} operator $\B_{\eps,i}$ for the set $I^\eps_i \setminus \bigcup_{\eps z_j \in I^{\eps}_i} B_{\eps^{\alpha}r_j}(\eps z_j)$ is content of Lemma~\ref{lem:O-EpsJohn}, {\color{red}provided we choose $\delta$ from \eqref{DeltaJedna} possibly even smaller to satisfy also \eqref{eq:DefDelta}}. We are now ready to define the restriction operator from $D$ to $D_\eps$ via 
\begin{align*}
R_\e(\vb u)\!:=\!\vb u-\!\!\!\!\!\!\!\sum_{z_j\in\Phi^\e(D)}\!\!\!\! (\beta_{\e,j}(\vb u)-\widetilde \B_{\e,j}(\div \beta_{\e,j}(\vb u)))\!-\!\sum_{i} (\vb b_{\e,i}(\vb u)\!-\!\B_{\e,i}(\div \vb b_{\e,i}(\vb u))),
\end{align*}
where the last sum runs over all boxes $I^\e_i$ and the functions were extended by $0$ outside their domain of definition. 
%nd $\B_{\e,i}(\div \beta_{\e,i}(u))$ and $\B_{\e,O^\e}(\div b_\e(u))$ has been extended by zero outside $B_{\tau r_i\e^\alpha}(\e r_i)$ and $O^\e$, respectively. 
This definition is essentially the same as in \cite{DieningFeireislLu}; note that we just replaced their operators $\B_{E_{\e,\! n}}$ by our operators $\B_{\e,i}$. Repeating the arguments shown in \cite[Section 3]{DieningFeireislLu}, by \eqref{eq:estimateB(u)}, \eqref{eq:estimateBeta(u)} and the fact that the boxes $I^\eps_i$ are disjoint, %two boxes $O^\e$ and $\widetilde{O^\e}$ do not overlap, 
we see that $R_\e(\vb u) \in W_0^{1,q}(D_\e)$ is well defined and satisfies{\color{red}
\begin{align*}
%R_\e(\vb u)\in W_0^{1,q}(D_\e),\quad 
\div R_\e(\vb u)=f \text{ in } D_\e,\quad \|R_\e(\vb u)\|_{W_0^{1,q}(D_\e)}\leq C\, \bigg(\e^{(\frac3q-1)(\alpha-\frac3m-\delta)-\frac3q (1+\delta)}+1\bigg)\|\vb u\|_{W_0^{1,q}(D)},
\end{align*}
where the constant $C>0$ is independent of $\e>0$. Note that due to the choice of $\delta$, the exponent of $\e$ on the right hand-site is non-negative, so we may bound $R_\e$ uniformly w.r.t.~$\e$. For $f\in L_0^q(D_\e)$ we define
\begin{align*}
\B_\e(f):=(R_\e\circ \B_D)(\tilde{f})
\end{align*}
and observe that we get the desired operator, namely that $\B_\e(f)\in W_0^{1,q}(D_\e)$, 
\begin{align*}
\div \B_\e(f)=f \text{ in } D_\e,\quad \textrm{and } \|\B_\e(f)\|_{W_0^{1,q}(D_\e)}\leq C\, \|f\|_{L^q(D_\e)}.
\end{align*}}%
This finishes the proof of Theorem \ref{thm:MainBog}.
\end{proof}

\section{Application to the Navier-Stokes equations}\label{sec:ApplNSE}

In this section, we will show the homogenization result for Navier-Stokes equations in a randomly perforated domain in the subcritical case $\alpha>3$. % and for the adiabatic exponent $\gamma>3$. 
The proof of such result in the case of periodically arranged holes was developed in a series of works~\cite{DieningFeireislLu,FeireislLu,LuSchwarzacher}, and can be split in two parts. First, using Bogovski\u{\i} operator we construct a good test function for the momentum equation, which leads to uniform in $\e$ estimates on the density as well as the velocity, subsequently providing the compactness. To identify the limiting ``effective'' equation, we need to construct a suitable cut-off function in order to compare the limiting equation with the equation in $D_\eps$. Since the rest of the proof does not refer in any way to location or structure of the holes, in particular it applies verbatim in our context, for that remaining part of the proof we only sketch the main steps. To shorten the exposition we sketch the argument only in the stationary case, following \cite{FeireislLu}. An analogous homogenization result holds also in the time-dependent setting -- see the statement and the proof of~\cite[Theorem 1.6]{LuSchwarzacher}.

\subsection{Test functions}

Before we formulate and show the homogenization result, we prove a modification of~\cite[Lemma 2.1]{LuSchwarzacher} in the random setting as the last ingredient in the proof of Theorem~\ref{thm:main}, which makes a reference to the randomness in the structure of the holes:
{\color{red}
\begin{lemma}\label{lm:cutoff}
Let $\alpha > 2$, $D\subset\R^3$ be a bounded $C^2$ star-shaped domain with $0 \in D$, and $(\Phi,\mathcal{R})=(\{z_i\},\{r_i\})$ be a marked Poisson point process with intensity $\lambda > 0$ and $r_i \ge 0$ with {\color{red}$\mathbb{E}(r_i^M) < \infty$ for $M=\max\{3,m\}$, where $m>3/(\alpha-2)$.} Then for any $1 < r < 3$ such that $(3-r)\alpha - 3 > 0$ and for almost every $\omega$ there exist a positive $\e_0(\omega)$ and a family of functions $\{ g_\eps\}_{\eps > 0} \subset W^{1,r}(D)$ such that for $0 < \e \leq \e_0$,
 \begin{equation}
  g_\e = 0 \quad \textrm{ in } \bigcup_{z_j \in \Phi^\eps(D)} B_{\eps^\alpha r_j}(\eps z_j), \qquad g_\e \to 1 \quad \textrm{ in } W^{1,r}(D) \textrm{ as } \eps \to 0
 \end{equation}
 and there is a constant $C>0$ such that
 \begin{equation}
  \| g_\e - 1 \|_{W^{1,r}(D)} \le C\e^{\sigma} \qquad \textrm{ with } \sigma := ((3-r)\alpha-3)/r. 
 \end{equation}
\end{lemma}

\begin{proof}
By $M>3/(\alpha-2)$ and Theorem \ref{thm:MainProbThm}, all the balls $\{B_{2\e^\alpha r_j}(\e z_j)\}_{z_j\in\Phi^\e(D)}$ are disjoint. Thus, there exist functions $g_\e\in C^\infty(D)$ such that
\begin{align*}
&0\leq g_\e\leq 1,\quad g_\e = 0 \textrm{ in } \bigcup_{z_j \in \Phi^\eps(D)} B_{\eps^\alpha r_j}(\eps z_j),\quad g_\e = 1 \textrm{ in } D\setminus\bigcup_{z_j \in \Phi^\eps(D)} B_{2\eps^\alpha r_j}(\eps z_j),\\
&\|\nabla g_\e\|_{L^\infty(B_{2\e^\alpha r_j}(\e z_j))}\leq C (\e^\alpha r_j)^{-1} \textrm{ for all } z_j\in\Phi^\e(D),
\end{align*}
where the constant $C>0$ is independent of $\e$ and $r_j$. Moreover, since $M\geq 3$, \eqref{eq:ergodic} yields 
$\lim\limits_{\eps \to 0} \eps^3 \sum_{z_j \in \Phi^\e(D)} r_j^3 = C$, thus implying
\begin{equation}\nonumber
 \biggl| \bigcup_{z_j \in \Phi^\e(D)} B_{2\eps^\alpha r_j}(\eps z_j) \biggr| \le |B_2| \eps^{3\alpha} \sum_{z_j \in \Phi^\e(D)} r_j^3 \le C \eps^{3(\alpha-1)}
\end{equation}
for $\eps > 0$ small enough. This together with direct calculation yields that for any $1<r<3$,
\begin{align*}
\|1-g_\e\|_{L^r(D)}\leq C\e^\frac{3(\alpha-1)}{r},\quad \|\nabla g_\e\|_{L^r(D)}\leq C\e^{\frac{(3-r)\alpha-3}{r}},
\end{align*}
which finally leads to
\begin{align*}
\|1-g_\e\|_{W^{1,r}(D)}=\|1-g_\e\|_{L^r(D)}+\|\nabla g_\e\|_{L^r(D)}\leq C\e^\sigma.
\end{align*}
\end{proof}}

%We will just sketch the proof since we would repeat the arguments from \cite{FeireislLu} almost word by word. We give the main issues to apply their methods in our situation. As a matter of fact, we also can proof the homogenization for $\gamma>2$ as done in \cite{DieningFeireislLu} as well as the homogenization for the time dependent case and $\gamma>6$ as done in \cite{LuSchwarzacher}.

\subsection{The Navier-Stokes equations, weak solutions, and the convergence result}
For $\eps > 0$, in the domain $D_\eps$ as in~\eqref{def:Domain} we consider the stationary Navier-Stokes equations for compressible viscous fluids 
%In this section, let $\Phi=\{z_i\}_{i\in\N}$ be as introduced in Section \ref{sec:Introd}. The points $z_i$ are seen as the centres of very small spherical holes $B_{r_i \e^\alpha}(\e z_i)$, where $\e>0$ and $\alpha>3$. Further, we consider the stationary Navier-Stokes equations in the domain $D_\e$:
\begin{align}
%\partial_t \rho + 
\div(\rho_\e \vb u_\e)&=0\hspace{3.8em}&&\text{ in } D_\e\label{eq:conteq},\\
%\partial_t(\rho u) + 
\div(\rho_\e \vb u_\e\otimes \vb u_\e) + \nabla p(\rho_\e) &= \div\mathbb{S}(\nabla \vb u_\e) + \rho_\e \vb f+\vb g&&\text{ in } D_\e\label{eq:momentum},\\
\vb u_\e&=0\hspace{3.8em} &&\text{ on } \partial D_\e,\label{eq:bdryval}
\end{align}
where $\mathbb{S}$ denotes the Newtonian viscous stress tensor of the form
\begin{align*}
\mathbb{S}(\nabla \vb u)=\mu \bigg(\nabla \vb u+\nabla^T \vb u-\frac23 \div(\vb u)\mathbb{I}\bigg)+\eta\div(\vb u)\mathbb{I},\quad \mu > 0, \quad \eta \ge 0,
\end{align*}
$p(\rho) = a\rho^{\gamma}$ denotes the pressure with $a > 0$ and the adiabatic exponent $\gamma \ge 1$, and $\vb f$ and $\vb g$ are external forces satisfying $\|\vb f\|_{L^\infty(\R^3;\R^3)} + \|\vb g\|_{L^\infty(\R^3;\R^3)} \le C$. 
We also fix the total mass
\begin{align*}
\int_{D_\e} \rho_\e=\mass>0
\end{align*}
independently of $\e>0$. 

%\subsection{Weak solutions}
% \footnote{adjust to time-dep wk sol}
\begin{defin}\label{def:wksol}
We call a couple $[\rho,\vb u]$ a \emph{renormalized finite energy weak solution} to equations \eqref{eq:conteq}-\eqref{eq:bdryval} if:
\begin{gather*}
\rho\geq 0 \text{ a.e. in $D_\e$,}\quad \int_{D_\e}\rho=\mass,\quad \rho\in L^{2\gamma}(D_\e),\quad \vb u\in W^{1,2}_0(D_\e);\\
%\int_{D_\e} \rho \vb u\cdot\nabla\psi=0,\\
\int_{D_\e} p(\rho)\div \phi+\rho \vb u\otimes \vb u:\nabla\phi-\mathbb{S}(\nabla \vb u):\nabla\phi+(\rho \vb f+\vb g)\cdot\phi=0
\end{gather*}
for all all test functions $\phi\in C_c^\infty(D_\e;\R^3),$ the energy inequality
\begin{align}\label{eq:energyineq}
\int_{D_\e} \mathbb{S}(\nabla \vb u):\nabla \vb u\leq \int_{D_\e} (\rho \vb f+\vb g)\cdot \vb u,
\end{align}
holds, and the zero extension $[\tilde{\rho},\tilde{\vb u}]$ satisfies in $\mathcal{D}^\prime(\R^3)$
\begin{align*}
\div(\tilde{\rho}\tilde{\vb u})=0,\quad \div(b(\tilde{\rho})\tilde{\vb u})+(\tilde{\rho}b^\prime(\tilde{\rho})-b(\tilde{\rho}))\div \tilde{\vb u}=0
\end{align*}
for any $b\in C([0,\infty))\cap C^1((0,\infty))$ such that there are constants
\begin{align*}
c>0,\quad \lambda_0<1,\quad -1<\lambda_1\leq\gamma-1
\end{align*}
with
\begin{align*}
b^\prime(s)\leq cs^{-\lambda_0} \text{ for } s\in (0,1],\quad b^\prime(s)\leq cs^{\lambda_1} \text{ for } s\in[1,\infty).
\end{align*}
\end{defin}

As announced above we show how to apply our results \emph{only in the stationary case}. For that let us first formulate the actual result:
%For the assumptions made above, we have the following result:
\begin{theorem}\label{thm:NSstat}
{\color{red}Assume $\alpha>3$.} Let $D\subset \R^3$ be a bounded star-shaped domain w.r.t.~the origin with $C^2$-boundary and let $(\Phi,\mathcal{R})=(\{z_j\},\{r_j\})$ be a marked Poisson point process with intensity $\lambda > 0$, and $r_j \ge 0$ with {\color{red}$\mathbb{E}(r_j^M) < \infty$, $M=\max\{3,m\}$, $m>3/(\alpha-3)$.} Farther let
% Let $D\subset \R^3$ be a bounded star-shaped domain with $C^2$-boundary with $0 \in D$ and let
\begin{align*}
\mass>0,\; \gamma>3.
\end{align*}
Then for almost every $\omega \in \Omega$ there exists $\eps_0=\eps_0(\omega) > 0$, such that the following holds: For $0 < \e < 1$ let $D_\e$ be as in \eqref{def:Domain} and let $[\rho_\e,\vb u_\e]$ be a family of renormalized finite energy weak solutions to \eqref{eq:conteq}-\eqref{eq:bdryval}. 
Then there is a constant $C>0$, which is independent of $\e$, such that 
\begin{align*}
\sup_{\e\in (0,\eps_0)} \|\tilde{\rho}_\e\|_{L^{2\gamma}(D)}+\|\tilde{\vb u}_\e\|_{W_0^{1,2}(D)} \leq C
\end{align*}
and, up to a subsequence, 
\begin{align*}
\tilde{\rho}_\e\weak \rho \text{ in } L^{2\gamma}(D),\quad \tilde{\vb u}_\e\weak \vb u \text{ in } W_0^{1,2}(D),
\end{align*}
where the limit $[\rho,\vb u]$ is a renormalized finite energy weak solution to the problem \eqref{eq:conteq}-\eqref{eq:bdryval} in the limit domain $D$.
\end{theorem}

\begin{rem}
{\color{red}The condition~$M>3/(\alpha-3)$ on the size of radii of the perforations is not just needed for technical purposes, but it is in a sense an optimal assumption. Indeed, one can show that in the case $m=3/(\alpha-3)$, for almost every realization of points and radii there is a sequence of $\eps \to 0$ such that for each such $\eps$ the rescaled radius $\eps^\alpha r_j$ of the largest ball in $D$ is of size $\eps^3$, i.e., $r_j \sim \eps^{3-\alpha}$. While \emph{one} large ball of size $\eps^3$ does not necessarily mean that the system should behave as in the critical case (which is expected to lead to a law of Brinkman type), nevertheless the presence of such large ball might change some of the properties of the system. Moreover, in the case $m < 3/(\alpha-3)$, the size of the largest ball would scale like $\eps^\nu$ with $\nu<3$, and there might be many balls of size at least $\eps^3$. 
}%
\end{rem}

% We find that the random position and radii of the holes have no influence to the convergence. This was known for the periodic and constant radii case, see, for instance, \cite{DieningFeireislLu}, and in the incompressible setting for random perforation also by \cite{GiuntiHofer}, \cite{Hillairet}. Schwarzacher and Lu gave in \cite{LuSchwarzacher} a convergence result for the compressible and random perforated case, but they still need that the holes are at least of order $\e$ far away from each other. With the results stated in Section \ref{sec:Model} and proven in Section \ref{sec:BogOp}, we are able to remove this assumption in a certain way.

% \subsection{Uniform bounds}
% {\Large from here, all is the same as in \cite{FeireislLu}; Sec. \ref{sec:Prob} are probabilistic results}\\

\begin{proof}[Sketch of the proof of Theorem \ref{thm:NSstat}]
We want to give uniform bounds for the velocity $\vb u_\e$ and the density $\rho_\e$ arising in the Navier-Stokes equations \eqref{eq:conteq}-\eqref{eq:bdryval}. First, by the energy inequality \eqref{eq:energyineq}, Korn's and Hölder's inequality, we have
\begin{align*}
\|\nabla \vb u_\e\|_{L^2(D_\e)}\lesssim \|\vb f\|_{L^\infty(D_\e)}\|\rho_\e\|_{L^\frac65(D_\e)}\|\vb u_\e\|_{L^6(D_\e)}+\|\vb g\|_{L^\infty(D_\e)} \|\vb u_\e\|_{L^6(D_\e)}.
\end{align*}
Since $\vb u_\e\in W_0^{1,2}(D_\e)$, we use Poincar\'e's inequality and Sobolev embedding to obtain $\|\vb u_\e\|_{L^6(D_\e)}\lesssim \|\nabla \vb u_\e\|_{L^2(D_\e)}$, which combined with the previous display yields
\begin{equation}
\begin{aligned}\label{eq:estimatevel}
\|\nabla \vb u_\e\|_{L^2(D_\e)}+\|\vb u_\e\|_{L^6(D_\e)} &\lesssim \|\vb f\|_{L^\infty(D_\e)}\|\rho_\e\|_{L^\frac65(D_\e)}+\|\vb g\|_{L^\infty(D_\e)}
\\
&\lesssim \|\rho_\e\|_{L^\frac65 (D_\e)}+1.
\end{aligned}
\end{equation}

We define a test function
\begin{align*}
\phi:=\B_\e\bigl(\rho_\e^\gamma-\langle \rho_\e^\gamma\rangle_{D_\e}\bigr),
\end{align*}
where $\langle \rho_\e^\gamma\rangle_{D_\e}:=|D_\e|^{-1}\int_{D_\e} \rho_\e^\gamma$ is the mean value of $\rho_\e^\gamma$ over the domain $D_\e$ and $\B_\e$ is the Bogovski\u{\i} operator constructed in Theorem \ref{thm:MainBog}. We remark that $\phi$ is well-defined due to the fact $\rho_\e^\gamma\in L^2(D_\e)$. By the properties of $\B_\e$, we obtain $\div \phi=\rho_\e^\gamma-\langle\rho_\e^\gamma\rangle_{D_\e}$ in $D_\e$ and 
\begin{align*}
\|\phi\|_{W_0^{1,2}(D_\e)}\lesssim \|\rho_\e^\gamma\|_{L^2(D_\e)}+\|\rho_\e^\gamma\|_{L^1(D_\e)}\lesssim \|\rho_\e\|_{L^{2\gamma}(D_\e)}^\gamma.
\end{align*}
Testing \eqref{eq:momentum} with $\phi$ yields
\begin{align*}
\int_{D_\e} p(\rho_\e)\rho_\e^\gamma =\sum_{j=1}^4 I_j,
\end{align*}
where the integrals $I_j$ are defined as
\begin{align*}
&I_1:=\int_{D_\e} p(\rho_\e)\langle\rho_\e^\gamma\rangle_{D_\e},\quad &&I_2:=\int_{D_\e}\mu\nabla \vb u:\nabla\phi+\bigg(\frac\mu3+\eta\bigg)\div \vb u_\e \div \phi,\\
&I_3:=-\int_{D_\e} \rho_\e \vb u_\e\otimes \vb u_\e :\nabla\phi,\quad &&I_4:=-\int_{D_\e} (\rho_\e \vb f+\vb g)\cdot \phi.
\end{align*}
By interpolation of Lebesgue spaces, we estimate $I_1$ by
\begin{align*}
|I_1|\leq \frac{1}{|D_\e|} \|\rho_\e\|_{L^\gamma(D_\e)}^{2\gamma}\leq \frac{1}{|D_\e|}\bigg(\|\rho_\e\|_{L^1(D_\e)}^{\theta_1}\|\rho_\e\|_{L^{2\gamma}(D_\e)}^{1-\theta_1}\bigg)^{2\gamma}\lesssim \|\rho_\e\|_{L^{2\gamma}(D_\e)}^{2\gamma(1-\theta_1)},
\end{align*}
where $\theta_1\in(0,1)$ is determined by
\begin{align*}
\frac1\gamma=\frac{\theta_1}{1}+\frac{1-\theta_1}{2\gamma}.
\end{align*}

The estimates for the remaining integrals are the same as in \cite{FeireislLu}, so we refer to~\cite[page 386]{FeireislLu} for details. Finally, we obtain
% \begin{align*}
$\|\rho_\e \|_{L^{2\gamma}(D_\e)}^{2\gamma}\lesssim \|\rho_\e\|_{L^{2\gamma}(D_\e)}^{2\gamma(1-\beta)}$
% \end{align*}
for some $\beta>0$, which yields
%\begin{align*}
$
\|\rho_\e\|_{L^{2\gamma}(D_\e)}\leq C
$.
% \end{align*}
In view of \eqref{eq:estimatevel}, we also have
% \begin{align*}
$\|\vb u_\e\|_{W_0^{1,2}(D_\e)}\leq C$, 
% \end{align*}
where the constant $C>0$ does not depend on $\e$. This completes the proof for the uniform bounds.

%\subsection{Equations in fixed domain}
In the following we want to identify the limiting equations. First, using the fact that $[\rho_\e, \vb u_\e]$ is a renormalized weak solution in $D_\eps$ we get that the zero extensions of $\rho_\e$ and $\vb u_\e$ solve
%In this section, we present the equations in the fixed domain $D$. We start with the continuity equation:
% Let the assumptions of Theorem \ref{thm:main} hold. Then the zero prolongation of the velocity and density satisfy
\begin{align*}
\div(\tilde{\rho}_\e\tilde{\vb u}_\e)=0,\quad \div(b(\tilde{\rho}_\e)\tilde{\vb u}_\e)+(\tilde{\rho}_\e b^\prime(\tilde{\rho}_\e)-b(\tilde{\rho}_\e))\div \tilde{\vb u}_\e=0 \text{ in } \mathcal{D}^\prime,
\end{align*}
where $b\in C([0,\infty))\cap C^1((0,\infty))$ is as in Definition \ref{def:wksol}. 

\medskip
Considering the momentum equation in the whole domain, we get an error $F_\eps$ on the right-hand side of the equation. Since the balls are tiny ($\alpha > 3$), this friction term is in the limit negligible. 
% Next, we consider the momentum equation. Here we show:
% \begin{theorem}\label{thm:extendedMomentEqu}
% Let the assumptions of theorem \ref{thm:main} hold. 
More precisely, the zero prolongations of the density and velocity satisfy
\begin{equation}\label{eq:cont}
\nabla p(\tilde{\rho}_\e)+\div(\tilde{\rho}_\e \tilde{\vb u}_\e\otimes \tilde{\vb u}_\e)-\div\mathbb{S}(\nabla \tilde{\vb u}_\e)=\tilde{\rho}_\e \vb f+\vb g+F_\e,
\end{equation}
where $F_\e$ is a distribution satisfying for all $\phi\in C_c^\infty(D)$
\begin{align*}
|\langle F_\e,\phi\rangle_{\mathcal{D}^\prime, \mathcal{D}}|\lesssim \e^\sigma \|\phi\|_{L^r(D)}+\e^\frac{3(\alpha-1)\sigma_0}{2(2+\sigma_0)}\|\nabla\phi\|_{L^{2+\sigma_0}(D)}
\end{align*}
with
\begin{align*}
\sigma:=\frac{\alpha-3}{4},\quad r:=\frac{12(\alpha-1)}{\alpha-3},\quad \sigma_0\in (0,\infty).
\end{align*}

To show this, we will use the cut-off function $g_\eps$ from Lemma~\ref{lm:cutoff}. For any test function $\phi\in C_c^\infty(D)$, we test the momentum equation in $D_\eps$ with $\phi g_\eps$ to get:
\begin{align*}
&\int_D \tilde{\rho}_\e\tilde{\vb u}_\e\otimes\tilde{\vb u}_\e : \nabla\phi+p(\tilde{\rho}_\e)
\div\phi-\mathbb{S}(\nabla\tilde{\vb u}_\e)
:\nabla\phi+(\tilde{\rho}_\e \vb f+\vb g)\cdot\phi\, dx\\
=&I_\e+\int_D \tilde{\rho}_\e\tilde{\vb u}_\e\otimes\tilde{\vb u}_\e : \nabla(g_\e\phi)+p(\tilde{\rho}_\e)
\div(g_\e\phi)-\mathbb{S}(\nabla\tilde{\vb u}_\e)
:\nabla(g_\e\phi)
\\
&\qquad\quad+(\tilde{\rho}_\e \vb f+\vb g)\cdot(g_\e\phi)\, dx\\
=&I_\e,
\end{align*}
%\begin{comment}
where we used that $g_\e\phi\in C_c^\infty(D)$ is an appropriate test function, and the term $I_\e$ is given by
\begin{align*}
I_\e:=\sum_{j=1}^4 I_{\e, j}:= &\int_D \tilde{\rho}_\e\tilde{\vb u}_\e\otimes\tilde{\vb u}_\e : (1-g_\e)\nabla\phi-\tilde{\rho}_\e\tilde{\vb u}_\e\otimes\tilde{\vb u}_\e : (\nabla g_\e\otimes\phi)\, dx\\
&+\int_D p(\tilde{\rho}_\e) (1-g_\e)\div\phi-p(\tilde{\rho}_\e)\nabla g_\e\cdot \phi\, dx\\
&+\int_D -\mathbb{S}(\nabla\tilde{\vb u}_\e):(1-g_\e)\nabla\phi+\mathbb{S}(\nabla \tilde{\vb u}
_\e):(\nabla g_\e\otimes\phi)\, dx\\
&+\int_D (\tilde{\rho}_\e \vb f+\vb g)\cdot (1-g_\e)\phi\, dx.
\end{align*}
Using the bounds on the cut-off function $g_\eps$, we combine the previous estimates on the density and velocity to prove~\eqref{eq:cont} (for details see~\cite[Proof of Proposition 2.2]{FeireislLu}).

% 
% Estimating each $I_{\e, j}$ separately, we get for $I_{\e, 1}$, using $\gamma>3$,
% \begin{align*}
% I_{\e,1}&\lesssim ||\rho_\e||_{L^{2\gamma}(D)} ||u_\e||_{L^6(D)}^2 \bigg(||(1-g_\e)\nabla\phi||_{L^2(D)}+||\nabla g_\e\otimes\phi||_{L^2(D)}\bigg)\\
% &\lesssim ||1-g_\e||_{L^\frac{2(2+\sigma_0)}{\sigma_0}(D)} ||\nabla\phi||_{L^{2+\sigma_0}(D)}+||\nabla g_\e||_{L^{r_1}(D)} ||\phi||_{L^{r_2}(D)},
% \end{align*}
% where
% %\begin{align*}
% $\sigma_0\in (0,\infty)$, $r_j\in (2,\infty)$, and $\frac{1}{r_1}+\frac{1}{r_2}=\frac12$. 
% % \end{align*}
% Using the estimates for $g_\e$, we infer
% \begin{align*}
% I_{\e, 1}\lesssim \e^\frac{3(\alpha-1)\sigma_0}{2(2+\sigma_0)}||\nabla\phi||_{L^{2+\sigma_0}(D)}+\e^{\frac{3(\alpha-1)}{r_1}-\alpha}||\phi||_{L^{r_2}(D)}.
% \end{align*}
% 
% Choosing
% \begin{align*}
% r_2:=\frac{12(\alpha-1)}{\alpha-3},
% \end{align*}
% 
% we get
% \begin{align*}
% \sigma:=\frac{3(\alpha-1)}{r_1}-\alpha=\frac{3(\alpha-1)}{2}-\alpha-\frac{3(\alpha-1)}{r_2}=\frac{\alpha-3}{2}-\frac{3(\alpha-1)}{r_2}=\frac{\alpha-3}{4}>0.
% \end{align*}
% Seeing that estimates for $I_{\e, 2}, I_{\e, 3}$ and $I_{\e, 4}$ follow in the same manner, we finish the proof of Theorem \ref{thm:extendedMomentEqu}.
%\end{comment}
% \end{proof}

% \subsection{The limiting equations}
% \subsubsection{Strong convergence of the velocity}

\medskip
By the uniform estimates on $\rho_\e$ and $\vb u_\e$, we get a subsequence (not relabeled) such that
\begin{align*}
\tilde{\rho}_\e\weak \rho \text{ in } L^{2\gamma}(D),\quad \tilde{\vb u}_\e\weak \vb u \text{ in } W_0^{1,2}(D).
\end{align*}
By compact Sobolev embedding, this yields
\begin{align*}
&\tilde{\vb u}_\e\to \vb u \text{ strongly in $L^q(D)$ for all } 1\leq q<6,\\
&\tilde{\rho}_\e \tilde{\vb u}_\e\weak \rho \vb u \text{ weakly in $L^q(D)$ for any } 1< q < \frac{6\gamma}{\gamma+3},\\
&\tilde{\rho}_\e\tilde{\vb u}_\e\otimes\tilde{\vb u}_\e\weak \rho \vb u \otimes \vb u \text{ weakly in $L^q(D)$ for all } 1<q<\frac{6\gamma}{2\gamma+3}.
\end{align*}
Letting $\e\to 0$ in equations \eqref{eq:conteq} and \eqref{eq:momentum}, we get the following equations in $\mathcal{D}^\prime(D)$:
\begin{align*}
\div(\rho \vb u)&=0,\\
\div(\rho \vb u\otimes \vb u)+\overline{p(\rho)}&=\div\mathbb{S}(\nabla \vb u)+\rho \vb f+\vb g,
\end{align*}
where $\overline{p(\rho)}$ is the weak limit of $p(\tilde{\rho}_\e)$ in $L^2(D)$. Moreover, the couple $[\rho,u]$ satisfies the renormalized equations. % stated in Remark \ref{rem:renormalized}. 
To finish the proof of Theorem \ref{thm:main}, we have to prove $\overline{p(\rho)}=p(\rho)$, arguing as in \cite[Section 2.4.2]{FeireislLu}.
\end{proof}

\noindent
{\bf Acknowledgement.} The authors were partially supported by the German Science Foundation DFG in context of
the Emmy Noether Junior Research Group BE 5922/1-1.  
  
%Einbinden der Bibliographie
\phantomsection %für korrekte Seitenzahl der Bib
%\addcontentsline{toc}{section}{Literatur}
%Name der Bib; siehe Lit.bib
%\bibliography{../Lit}   % name your BibTeX data base

 \bibliographystyle{amsplain}
 \bibliography{Lit}

\end{document}